\documentclass{amsart}

\usepackage{amsmath,amssymb,eucal,amsthm}

\newtheorem{theorem}{Theorem}
\newtheorem{lemma}[theorem]{Lemma}
\newtheorem{proposition}[theorem]{Proposition}
\newtheorem{corollary}[theorem]{Corollary}

\newtheorem{example}[theorem]{Example}

\newtheorem{remark}[theorem]{Remark}

\allowdisplaybreaks
\numberwithin{equation}{section}

\def\bs#1{\mathbf{#1}}
\def\fs{\Psi}

\def\prs{\Delta_+}
\def\gr{{\frak g}}
\def\abs#1{\lvert#1\rvert}
\def\norm#1{\lVert#1\rVert}

\begin{document}

\title[Zeta-functions of weight lattices of Lie groups]{Functional relations for zeta-functions of weight lattices of Lie groups of type $A_3$}

\author{Yasushi Komori, Kohji Matsumoto and Hirofumi Tsumura} 

\address{Y.\,Komori:\ Department of Mathematics, Rikkyo University, Nishi-Ikebukuro, Toshima-ku, Tokyo 171-8501, Japan\\
e-mail: komori@rikkyo.ac.jp}
\address{K.\,Matsumoto:\ Graduate School of Mathematics, Nagoya University, Chikusa-ku, Nagoya 464-8602, Japan\\
Email: kohjimat@math.nagoya-u.ac.jp} 
\address{H.\,Tsumura:\ Department of Mathematics and Information Sciences, Tokyo Metropolitan University, Hachioji, Tokyo 192-0397, Japan\\
e-mail: tsumura@tmu.ac.jp} 
\maketitle

\begin{abstract}
We study zeta-functions of weight lattices 
of compact connected semisimple Lie groups of type $A_3$. Actually we consider zeta-functions of $SU(4)$, $SO(6)$ and $PU(4)$, and give some functional relations and new classes of evaluation formulas for them.
\end{abstract}

\baselineskip 16pt

\section{Introduction}\label{sec-1}

For any semisimple Lie algebra ${\frak g}$, the Witten zeta-function 
$\zeta_W(s;{\frak g})$ is defined by
\begin{equation}
\zeta_W(s;{\frak g})=\sum_{\rho}\,\left({\rm dim}\,\rho\right)^{-s}, \label{eq-1-1}
\end{equation}
where $s \in \mathbb{C}$ and $\rho$ runs over all finite dimensional irreducible representations of $\gr$.    This was formulated by Zagier (see \cite[Section 7]{Za}),
who was inspired by Witten's work \cite{Wi}. 
Witten's motivation of introducing the above zeta-functions is to
express the volumes of certain moduli spaces in terms of special values of
$\zeta_W(s;{\frak g})$.
The result is called Witten's volume formula, and from which it can be shown that
\begin{align}
  \label{1-2}
    \zeta_W(2k;{\frak g})=C_{W}(2k,{\frak g})\pi^{2kn}
\end{align}
for any $k\in{\Bbb N}$, where 
$C_{W}(2k,{\frak g})\in{\Bbb Q}$ (see \cite[Theorem, p.506]{Za}).
The explicit value of $C_{W}(2k,{\frak g})$ was not determined in their work.

Gunnells and Sczech introduced a method to compute $C_{W}(2k,{\frak g})$ explicitly 
(see \cite{GS}).    The theory of Szenes (\cite{Sz98}, \cite{Sz03}) also gives a
different algorithm of computing $C_{W}(2k,{\frak g})$.

Let $r$ be the rank of ${\frak g}$, $\Delta({\frak g})$ the root system
corresponding to ${\frak g}$, and $n$ the number of positive roots belonging to
$\Delta({\frak g})$.   In \cite{KMT, KMTpja, KM2, MT}, the authors defined the
zeta-function $\zeta_r({\bf s};\Delta({\frak g}))$ of the root system
$\Delta({\frak g})$, where ${\bf s}=(s_i)\in{\Bbb C}^n$ (see Section \ref{sec-2}).   
This may be regarded as
a multi-variable version of $\zeta_W(s;{\frak g})$ (see also survey papers
\cite{KMT-Mem}, \cite{KMT-Bessatsu}).
The authors further introduced a generalization of Bernoulli polynomials
associated with root systems in \cite{KMTpja,KM5,KM3}.
Using these tools, we can generalize \eqref{1-2} 
(\cite[Theorem 4.6]{KM3}) and express $C_{W}(2k,{\frak g})$ in terms of Bernoulli
polynomials of root systems (see also \cite{KMT-Mem}), hence gives another algorithm
for computing $C_{W}(2k,{\frak g})$.     Moreover we can give various functional 
relations for zeta-functions of root systems (see 
\cite{KMT,KMTpja,KMTJC,KM5,KM3,KM4,MT}; we will discuss this matter further in 
the next section).

More recently, the authors defined zeta-functions of weight lattices of 
compact connected semisimple Lie groups (see \cite{KMT-Lie}). 
If the group is simply-connected, these zeta-functions coincide with 
ordinary zeta-functions of root systems of associated Lie algebras.
We considered the general connected (but not necessarily 
simply-connected) case and proved a result analogous to \eqref{1-2} for these 
zeta-functions, and further prove functional relations among them. 
The present paper is a continuation of \cite{KMT-Lie}, and we study zeta-functions
of lattices of Lie groups whose associated Lie algebras are of type $A_3$.
The reason why we treat the case $A_3$ will be mentioned in Section \ref{sec-2}.

Throughout this paper, let ${\Bbb N}$ be the set of positive integers, ${\Bbb N}_0={\Bbb N}\cup\{0\}$,
${\Bbb Z}$ the ring of rational integers, ${\Bbb Q}$ the rational number field,
${\Bbb R}$ the real number field, and ${\Bbb C}$ the complex number field.

\ 

\section{Functional relations: a motivation}
\label{sec-1.5}

In this section we comment our motivation on the study of functional 
relations.\footnote{The contents of this section was given            
in the talk of the second-named author on the problem session of the conference.}
The Euler-Zagier $r$-ple sum is defined by
\begin{align}\label{15-1}
\zeta_{EZ,r}({\bf s})=\sum_{m_1,\ldots,m_r=1}^{\infty}\frac{1}{m_1^{s_1}
(m_1+m_2)^{s_2}\cdots(m_1+\cdots+m_r)^{s_r}},
\end{align}
where ${\bf s}=(s_1,\ldots,s_r)\in{\Bbb C}^r$ (see \cite{Ho}, \cite{Za}).
Special values of \eqref{15-1} for ${\bf s}={\bf k}$, where 
${\bf k}=(k_1,\ldots,k_r)\in {\Bbb N}^r$ with $k_r\geq 2$, are known to be
important in various fields of mathematics.   Euler already obtained the following 
relations among those values in the case $r=2$:
The harmonic product relation
\begin{align}\label{15-2}
\zeta(k_1)\zeta(k_2)=\zeta_{EZ,2}(k_1,k_2)+\zeta_{EZ,2}(k_2,k_1)+\zeta(k_1+k_2),
\end{align}
where $\zeta(s)=\zeta_{EZ,1}(s)$ is the Riemann zeta-function, and the sum formula
\begin{align}\label{15-3}
\sum_{j=2}^{k-1}\zeta_{EZ,2}(k-j,j)=\zeta(k)
\end{align}
for $k\in{\Bbb N}$, $k\geq 3$.   After the discovery of the importance of
Euler-Zagier sums (around 1990, according to the work of Drinfel'd, Goncharov,
Kontsevich, Hoffman and Zagier), many people began to search for various 
relations among special values of \eqref{15-1}, and indeed a lot of relations
have been discovered.

Around 2000, the second-named author raised a question: are those relations valid
only at positive integers, or valid also continuously at other values?

In fact, it is easy to see that \eqref{15-2} is valid for any complex numbers
$s_1,s_2$ except for singularities, that is
\begin{align}\label{15-4}                                                              
\zeta(s_1)\zeta(s_2)=\zeta_{EZ,2}(s_1,s_2)+\zeta_{EZ,2}(s_2,s_1)+\zeta(s_1+s_2).       
\end{align}
Therefore the harmonic product relation is actually a ``functional relation''.
So far, except for \eqref{15-4} and its relatives, no other such functional
relations among Euler-Zagier sums has been discovered.
(In the double zeta case, the functional equation \cite{KMT-Debrecen} is known,
but it is not a formula which interpolates some known value-relation.)

However, when we consider more extended classes of multiple series, we can
find a lot of functional relations!   The first examples were reported by the
third-named author \cite{TsC}, in which functional relations among $\zeta(s)$
and the Tornheim double sum
\begin{align}\label{15-5}
\zeta_{MT,2}(s_1,s_2,s_3)=\sum_{m,n=1}^{\infty}\frac{1}{m^{s_1}n^{s_2}
(m+n)^{s_3}}
\end{align}
are proved.   Those relations can be regarded as functional relations among
zeta-functions of root systems, because $\zeta_{MT,2}(s_1,s_2,s_3)$ coincides with
the zeta-function of the root system $\Delta({\frak su}(3))$.

It is known that irreducible root systems are classified as types $X_r$
(where $X$ is one of $A,B,C,D,E,F,G$) by the Killing-Cartan theory.   When $\Delta({\frak g})$
is of type $X_r$, we denote the corresponding zeta-function as
$\zeta_r({\bf s};X_r)$.   Using this notation, we see that \eqref{15-5} is
$\zeta_2({\bf s};A_2)$.

Various other functional relations have then been proved in several articles of the
authors: on $A_2$ (\cite{KM3}), 
on $A_3$ (\cite{MT}, \cite{KMT}, \cite{KMTJC}, \cite{KMT-Z}), 
on $B_2=C_2$ (\cite{KMTJC}, \cite{KM5}), on $B_3$ and $C_3$ (\cite{KMTJC}), 
and on $G_2$ (\cite{KM4}).
A general treatment on the theory of functional relations is given in
\cite{KM3}, \cite{KM5}.
Those functional relations in fact include various known value-relations among
special values of Euler-Zagier sums, and also include Witten's formula \eqref{1-2}
in several cases.

Moreover, functional relations also exist among zeta-functions of lattices of
(not necessarily simply-connected) Lie groups in the sense of \cite{KMT-Lie}.
In \cite{KMT-Lie}, we proved functional relations among zeta-functions whose 
associated root systems are of type $A_2$ or $C_2$.

In this paper we study zeta-functions of 
weight lattices of compact connected semisimple Lie groups of  type $A_3$. 
More precisely, we consider zeta-functions of $SU(4)$, $SO(6)$ and $PU(4)$.

In Section \ref{sec-2}, we recall the definition of those zeta-functions. In Section 
\ref{sec-3}, we prepare some lemmas which will be necessary later. 
Then in the remaining sections
we prove some functional relations for those zeta-functions which are the 
main results in this paper (see Theorems \ref{P-FR-A_3}, \ref{P-FR-A_3-2}, \ref{T-SO6}, \ref{T-6-3} and \ref{T-6-4}). 
Moreover we give new classes of evaluation formulas for these zeta-functions in terms 
of the Riemann zeta-function (see Propositions \ref{P-new} and \ref{P-SO6}, and 
Examples \ref{Exam-A3} and \ref{Exam-SO6}) and the Dirichlet $L$-function with
the primitive character of conductor 4 (see Proposition \ref{P-new-2} and Examples \ref{Ex-6-1} and \ref{Ex-6-2}).

\

\section{Zeta-functions of weight lattices}\label{sec-2}

In this section, we recall the definition and some properties of zeta-functions of weight lattices of compact connected semisimple Lie groups which we considered in our previous paper \cite[Section 3]{KMT-Lie}. 

We first prepare the same notation as in \cite{KM5,KM2,KM3} 
(see also \cite{KMT,KMTpja,KMT-Mem,KM4}). 
Let $V$ be an $r$-dimensional real vector space equipped with an inner product $\langle \cdot,\cdot\rangle$.
The norm $\norm{\cdot}$ is defined by $\norm{v}=\langle v,v\rangle^{1/2}$.
The dual space $V^*$ is identified with $V$ via the inner product of $V$.
Let $\Delta$ be a finite reduced root system which may not be irreducible, and
$\fs=\{\alpha_1,\ldots,\alpha_r\}$ its fundamental system.
We fix 
$\Delta_+$ and $\Delta_-$ as the set of all positive roots and negative roots respectively.
Then we have a decomposition of the root system $\Delta=\Delta_+\coprod\Delta_-$ .
Let $Q=Q(\Delta)$ be the root lattice, $Q^\vee$ the coroot lattice,
$P=P(\Delta)$ the weight lattice, $P^\vee$ the coweight lattice,
and $P_+$ the set of integral dominant weights 
defined by
\begin{align}
&  Q=\bigoplus_{i=1}^r\mathbb{Z}\,\alpha_i,\qquad
  Q^\vee=\bigoplus_{i=1}^r\mathbb{Z}\,\alpha^\vee_i,\\
& P=\bigoplus_{i=1}^r\mathbb{Z}\,\lambda_i, \qquad 
  P^\vee=\bigoplus_{i=1}^r\mathbb{Z}\,\lambda^\vee_i,\\
&   P_+=\bigoplus_{i=1}^r\mathbb{N}_0\,\lambda_i, 
\end{align}
respectively, where the fundamental weights $\{\lambda_j\}_{j=1}^r$
and
the fundamental coweights $\{\lambda_j^\vee\}_{j=1}^r$
are the dual bases of $\fs^\vee$ and $\fs$
satisfying $\langle \alpha_i^\vee,\lambda_j\rangle=\delta_{ij}$ and $\langle \lambda_i^\vee,\alpha_j\rangle=\delta_{ij}$ respectively.
Let
\begin{equation}
\rho=\frac{1}{2}\sum_{\alpha\in\Delta_+}\alpha=\sum_{j=1}^r\lambda_j
\label{rho-def}
\end{equation}
be the lowest strongly dominant weight.
Let $\sigma_\alpha$ be the reflection with respect to a root $\alpha\in\Delta$ defined as 
\begin{equation}
 \sigma_\alpha:V\to V, \qquad \sigma_\alpha:v\mapsto v-\langle \alpha^\vee,v\rangle\alpha.
\end{equation}
For a subset $A\subset\Delta$, let
$W(A)$ be the group generated by reflections $\sigma_\alpha$ for all $\alpha\in A$. In particular, $W=W(\Delta)$ is the Weyl group, and 
$\{\sigma_j=\sigma_{\alpha_j}\,|\,1\leq j \leq r\}$ generates $W$. 

Let $\widetilde{G}$ be a simply-connected compact connected semisimple Lie group,
and ${\frak g}={\rm Lie}(\widetilde{G})$.
There is a one-to-one correspondence between 
a compact connected semisimple Lie group $G$ whose universal covering group is
$\widetilde{G}$, and a lattice $L$ with 
$Q(\Delta({\frak g}))\subset L\subset P(\Delta({\frak g}))$ up to 
automorphisms 
by taking  $L=L(G)$ as the weight lattice of $G$.
Let $L_+=P_+\cap L$. 

Now we recall the definition of the zeta-function of the weight lattice $L=L(G)$ of the semisimple Lie group $G$, that is,
\begin{equation}
  \zeta_r(\bs{s},\bs{y};G)=
  \zeta_r(\bs{s},\bs{y};L;\Delta):=\sum_{\lambda\in L_++\rho}
    e^{2\pi i\langle\bs{y},\lambda\rangle}\prod_{\alpha\in\Delta_+}
    \frac{1}{\langle\alpha^\vee,\lambda\rangle^{s_\alpha}},
\label{Lerch-zeta-L}
\end{equation}
where $\bs{y}\in V$.
Note that this zeta-function can be continued 
meromorphically to $\mathbb{C}^{n}$.
When ${\bf y}={\bf 0}$, we sometimes write this zeta-function as $\zeta_r({\bf s};G)$ or
$\zeta_r(\bs{s};L;\Delta)$ for brevity.
It is to be noted that if $G=\widetilde{G}$, then $L=P$ and 
$\zeta_r({\bf s};\widetilde{G})$ coincides with $\zeta_r({\bf s};{\frak g})$, also written as $\zeta_r({\bf s};\Delta({\frak g}))$ and $\zeta_r({\bf s};X_r)$ when $\Delta({\frak g})$ is of  type $X_r$,
which is called the zeta-function of the root system of  type $X_r$ studied in our previous papers (see, for example, \cite{KMT,KMT-Mem,KM2,KM3}).

In the present paper we concentrate our attention on the case of  type $A_3$, so we
write down the explicit form of zeta-functions in this case.
Let $\Delta=\Delta(A_3)$ with $\Psi=\{ \alpha_1,\,\alpha_2,\,\alpha_3\}$, 
$\Delta_{+}=\{ \alpha_1,\,\alpha_2,\,\alpha_3,\,\alpha_1+\alpha_2,\,\alpha_2+\alpha_3, 
\,\alpha_1+\alpha_2+\alpha_3\}$, 
$P=\sum_{j=1}^{3}\mathbb{Z}\lambda_j$ and $Q=\sum_{j=1}^{3}\mathbb{Z}\alpha_j$.
It is known that $P/Q\simeq \mathbb{Z}/4\mathbb{Z}$.
Therefore there is a unique intermediate lattice $L_1$ with
$P\supsetneq L_1 \supsetneq Q$, satisfying $(L_1:Q)=2$.
The group corresponding to $P$ (resp. $Q$) is $SU(4)$ (resp. $PU(4)$).   The group
$G=G(L_1)$ is $SU(4)/\{\pm 1\}$, which is known to be isomorphic to $SO(6)$.
We know (for the details, see \cite[Example 4.3]{KMT-Lie}) that
\begin{align}\label{zeta-A3-P}                                                         
 & \zeta_3 ({\bf s},\bs{y};SU(4))=\zeta_3 ({\bf s},\bs{y};P;A_3) \\                   
   & =\sum_{m_1,m_2,m_3=1}^\infty                                                     
     \frac{e^{2\pi i\langle\bs{y},m_1\lambda_1+m_2\lambda_2+m_3\lambda_3\rangle}}
   {m_1^{s_1}m_2^{s_2}m_3^{s_3}(m_1+m_2)^{s_4}                                       
     (m_2+m_3)^{s_5}(m_1+m_2+m_3)^{s_6}}, \notag                                      
\end{align}
with
$$
\lambda_1=\frac{3}{4}\alpha_1+\frac{1}{2}\alpha_2+\frac{1}{4}\alpha_3,
\quad
\lambda_2=\frac{1}{2}\alpha_1+\alpha_2+\frac{1}{2}\alpha_3,
\quad
\lambda_3=\frac{1}{4}\alpha_1+\frac{1}{2}\alpha_2+\frac{3}{4}\alpha_3.
$$
Note that $\zeta_3 ({\bf s},\bs{0};SU(4))=\zeta_3 ({\bf s};A_3)$.
Further we have
\begin{align}\label{zeta-A3-L}                                                          
  &\zeta_3({\bf s},\bs{y};SO(6))=\zeta_3 ({\bf s},\bs{y};L_1;A_3)\\                    
  &=\sum_{m_1,m_2,m_3=1 \atop m_1\equiv m_3\,(\text{mod}\,2)}^\infty                    
     \frac{e^{2\pi i\langle\bs{y},m_1\lambda_1+m_2\lambda_2+m_3\lambda_3\rangle}}
   {m_1^{s_1}m_2^{s_2}m_3^{s_3}(m_1+m_2)^{s_4}                                         
     (m_2+m_3)^{s_5}(m_1+m_2+m_3)^{s_6}}, \notag \\
   &\zeta_3({\bf s},\bs{y};PU(4))=\zeta_3 ({\bf s},\bs{y};Q;A_3)\label{zeta-A3-Q}                                                          
\\                       
 &=\sum_{m_1,m_2,m_3=1 \atop m_1+2m_2+3m_3 \equiv 2\ (\text{mod}\ 4)}^\infty            
     \frac{e^{2\pi i\langle\bs{y},m_1\lambda_1+m_2\lambda_2+m_3\lambda_3\rangle}}
    {m_1^{s_1}m_2^{s_2}m_3^{s_3}(m_1+m_2)^{s_4}  
     (m_2+m_3)^{s_5}(m_1+m_2+m_3)^{s_6}}.\notag                                        
\end{align}
In particular,
\begin{align}
& {\zeta}_3({\bf s},\lambda_1^\vee;SU(4))={\zeta}_3({\bf s},\lambda_1^\vee;P;A_3) \label{A3-lam1}\\
& \quad =\sum_{l,m,n=1}^\infty \frac{i^{3l+2m+n}}{l^{s_1}m^{s_2}n^{s_3}(l+m)^{s_4}(m+n)^{s_5}(l+m+n)^{s_6}},\notag\\
& {\zeta}_3({\bf s},\lambda_2^\vee;SU(4))={\zeta}_3({\bf s},\lambda_2^\vee;P;A_3) \label{A3-lam2}\\
& \quad =\sum_{l,m,n=1}^\infty \frac{(-1)^{l+n}}{l^{s_1}m^{s_2}n^{s_3}(l+m)^{s_4}(m+n)^{s_5}(l+m+n)^{s_6}},\notag\\
& {\zeta}_3({\bf s},\lambda_3^\vee;SU(4))={\zeta}_3({\bf s},\lambda_3^\vee;P;A_3) \label{A3-lam3}\\
& \quad =\sum_{l,m,n=1}^\infty \frac{i^{l+2m+3n}}{l^{s_1}m^{s_2}n^{s_3}(l+m)^{s_4}(m+n)^{s_5}(l+m+n)^{s_6}}.\notag
\end{align}
Note that 
\begin{align}
& {\zeta}_3((s_1,s_2,s_3,s_4,s_5,s_6),\lambda_1^\vee;SU(4)) ={\zeta}_3((s_3,s_2,s_1,s_5,s_4,s_6),\lambda_3^\vee;SU(4)), \label{change-1}\\
& {\zeta}_3((s_1,s_2,s_3,s_4,s_5,s_6),\lambda_2^\vee;SU(4)) ={\zeta}_3((s_3,s_2,s_1,s_5,s_4,s_6),\lambda_2^\vee;SU(4)),\label{change-2}
\end{align}
as well as
\begin{align}
{\zeta}_3((s_1,s_2,s_3,s_4,s_5,s_6);A_3) ={\zeta}_3((s_3,s_2,s_1,s_5,s_4,s_6);
A_3). \label{change-0}
\end{align}

There are several reasons why the study on the case $A_3$ deserves one paper.
First, since the case $A_2$ was studied in \cite{KMT-Lie}, it is a natural
continuation.
Second, since the functional relation for $\zeta_3 ({\bf s};A_3)$ given in
\cite{KMTJC} is restricted to the case of even integers, we supply a more general
result (Theorem \ref{P-FR-A_3-2}) here.
The third reason is that Lie algebras of type $A_r$ are the most interesting in view
of the theory of \cite{KMT-Lie}, because 
$|P(\Delta(A_r))/Q(\Delta(A_r))|\to\infty$ as $r\to\infty$ (hence there are many
intermediate lattices between $P$ and $Q$), while$|P/Q|$ remains small for any
Lie algebras of other types (see Bourbaki \cite{Bour}).

Lastly in this section
we quote here one more general result, which is a generalization and a 
refinement of \eqref{1-2}.

\begin{theorem}{\rm \cite[Theorem 3.2]{KMT-Lie}} \label{thm:W-Z} 
  For a compact connected semisimple Lie group $G$, let
  $\Delta=\Delta(G)$ be its root system, 
  and $L=L(G)$ be its weight lattice.
  Let
  $\mathbf{k}=(k_\alpha)_{\alpha\in\prs}\in\mathbb{N}^{n}$ {\rm
    ($n=|\Delta_+|$)} satisfying $k_\alpha=k_\beta$ whenever
  $\norm{\alpha}=\norm{\beta}$.  Let
  $\kappa=\sum_{\alpha\in\Delta_+}{2k_\alpha}$.  Then we have
for $\nu\in P^\vee/Q^\vee$,
\begin{equation}
  \label{eq:vol_formula}
  \begin{split}
    \zeta_r& (2\mathbf{k},\nu;G) =\zeta_r(2\mathbf{k},\nu;L;\Delta)\\
    & =\frac{(-1)^{n}}{\abs{W}}
    \biggl(\prod_{\alpha\in\Delta_+}
    \frac{(2\pi\sqrt{-1})^{2k_\alpha}}{(2k_\alpha)!}\biggr)
    \mathcal{P}(2\mathbf{k},\nu;L;\Delta) \in \mathbb{Q}\cdot 
    \pi^{\kappa},
  \end{split}
\end{equation}
where $\mathcal{P}(2\mathbf{k},\nu;L;\Delta)$ is the Bernoulli function associated
with $L$, defined in \cite{KMT-Lie}.
\end{theorem}

Note that when $L=P$, \eqref{eq:vol_formula} coincides with our previous result 
in \cite[Theorem 4.6]{KM3}. 

As an example, here we apply this Theorem to the case of $PU(4)$.

\begin{example} \rm
The generating function of $\mathcal{P}(\mathbf{k},\mathbf{y};A_3)$ has been given
in \cite[Example 2]{KM5}.   Therefore, 
by \cite[(3.8)]{KMT-Lie}, we have
\begin{multline}
  \mathcal{P}((2,2,2,2,2,2),\mathbf{0};Q;A_3)
  \\
  \begin{aligned}
    &=
    \frac{1}{4}(
    \mathcal{P}((2,2,2,2,2,2),\mathbf{0};A_3)
    -\mathcal{P}((2,2,2,2,2,2),\lambda_1^\vee;A_3)
    \\
    &\qquad+\mathcal{P}((2,2,2,2,2,2),\lambda_2^\vee;A_3)
    -\mathcal{P}((2,2,2,2,2,2),\lambda_3^\vee;A_3)
    )
    \\
    &=\frac{1103}{96888422400},
  \end{aligned}
\end{multline}
because $2\rho=3\alpha_1+4\alpha_2+3\alpha_3$ and hence
$\langle\mathbf{0},2\rho\rangle=0$,
$\langle\lambda_1^\vee,2\rho\rangle=3$,
$\langle\lambda_2^\vee,2\rho\rangle=4$,
$\langle\lambda_3^\vee,2\rho\rangle=3$.
Therefore by Theorem \ref{thm:W-Z}, we obtain
\begin{equation}
  \zeta_3((2,2,2,2,2,2),\mathbf{0};PU(4))=
  \frac{1103\pi^{12}}{145332633600}. \label{val-PU4}
\end{equation}
\end{example}

\section{Some preparatory lemmas}\label{sec-3}

In this section, we give explicit functional relations for double polylogarithms (see Lemma \ref{Lem-5-1} and Corollary \ref{L-MT}). By use of these results, we give certain functional relations for triple zeta-functions of weight lattices in the next section.

First we quote the following two lemmas. Let 
$\phi(s)=\sum_{n\geq 1}(-1)^n n^{-s}=\left(2^{1-s}-1\right)\zeta(s)$ and $\varepsilon_m=\{ 1+(-1)^m\}/2$ for $m\in \mathbb{Z}$. 
Let $\{ B_m(X)\}$ be the Bernoulli polynomials defined by 
$$\frac{te^{Xt}}{e^t-1}=\sum_{m=0}^\infty B_m(X)\frac{t^m}{m!}.$$

\begin{lemma}{\rm (\cite[Lemma 9.1]{KM5}, \cite[Lemma 2.1]{MNO})} \label{L-4-3} 
Let $c\in[0,2\pi) \subset \mathbb{R}$, and
$h:\mathbb{N}_{0} \to \mathbb{C}$ be a function (which may depend on $c$). 
Then, for $p \in \mathbb{N}$,
\begin{equation}
\begin{split}
& \sum_{j=0}^{p} \phi(p-j) \varepsilon_{p-j}\sum_{\xi=0}^{j} h(j-\xi)\frac{(i(c-\pi))^{\xi}}{\xi!} =-\frac{1}{2}\sum_{\xi=0}^{p}h(p-\xi)\frac{(2\pi i)^{\xi}}{\xi!}B_{\xi}\left(\left\{ \frac{c}{2\pi}\right\} \right), 
\end{split}
\label{eq-5-6-2}
\end{equation}
and 
\begin{align}
& \sum_{j=0}^{p} \phi(p-j) \varepsilon_{p-j}\sum_{\xi=0}^{j} h(j-\xi)\frac{(i\pi)^{\xi}}{\xi!} =\sum_{\nu=0}^{[p/2]}\zeta(2\nu)h(p-2\nu)-\frac{i\pi}{2}h(p-1), \label{eq-MNOT}
\end{align}
where $[x]$ is the integer part of $x\in \mathbb{R}$ and $\{ x\}=x-[x]$.
\end{lemma}

The next lemma is the key to proving functional relations for zeta-functions. For $h \in \mathbb{N}$, let
\begin{align*}
& {\frak C}:=\left\{ C(l) \in \mathbb{C}\,|\, l \in \mathbb{Z},\ l \not=0 \right\}, \\
& {\frak D}:=\left\{ D(N;m;\eta) \in \mathbb{R}\,|\, N,m,\eta \in \mathbb{Z},\ N \not=0,\ m \geq 0,\ 1 \leq \eta \leq h\right\}, \\
& {\frak A}:=\{ a_\eta \in \mathbb{N}\,|\,1 \leq \eta \leq h\}
\end{align*}
be sets of numbers indexed by integers, and let
\begin{equation*}
\binom{x}{k}:=
\begin{cases}
\frac{x(x-1)\cdots(x-k+1)}{k!} & \ \ (k \in \mathbb{N}),\\
\ \ 1 & \ \ (k=0).
\end{cases}
\end{equation*}

\begin{lemma}{\rm \cite[ Lemma 6.2]{KMTJC}} \label{L-4-2} With the above notation, we assume that the infinite series appearing in 
\begin{equation}
\begin{split}
\sum_{N \in \mathbb{Z} \atop N \not=0}(-1)^{N}C(N)e^{iN\theta} & -2\sum_{\eta=1}^{h}\sum_{k=0}^{a_\eta}\phi(a_\eta-k)\varepsilon_{a_\eta-k} \\
& \ \ \times \sum_{\xi=0}^{k}\left\{ \sum_{N \in \mathbb{Z} \atop N \not=0}(-1)^N D(N;k-\xi;\eta)e^{iN\theta}\right\}\frac{(i\theta)^\xi}{\xi!} 
\end{split}
 \label{eq-4-3}
\end{equation}
are absolutely convergent for $\theta \in [-\pi,\pi]$, and that (\ref{eq-4-3}) is a constant function for $\theta \in (-\pi,\pi)$. Then, for $d \in \mathbb{N}$, 
\begin{align}
& \sum_{N \in \mathbb{Z} \atop N \not= 0}\frac{(-1)^{N}C(N)e^{iN\theta}}{N^d} 
=2\sum_{\eta=1}^{h}\sum_{k=0}^{a_\eta}\phi(a_\eta-k)\varepsilon_{a_\eta-k} 
\label{eq-4-4} \\
& \ \ \ \ \ \times \sum_{\xi=0}^{k}\bigg\{ \sum_{\omega=0}^{k-\xi}\binom{\omega+d-1}{\omega}(-1)^{\omega} \sum_{m \in \mathbb{Z} \atop m \not= 0}\frac{(-1)^m D(m;k-\xi-\omega;\eta)e^{im\theta}}{m^{d+\omega}}\bigg\}\frac{(i\theta)^\xi}{\xi!} \notag \\
& \ -2\sum_{k=0}^{d}\phi(d-k)\varepsilon_{d-k} \sum_{\xi=0}^{k}\bigg\{ \sum_{\eta=1}^{h} \sum_{\omega=0}^{a_\eta-1}\binom{\omega+k-\xi}{\omega}(-1)^{\omega}\notag \\
& \hspace{1in} \times \sum_{m \in \mathbb{Z} \atop m \not=0}\frac{D(m;a_\eta-1-\omega;\eta)}{m^{k-\xi+\omega+1}}\bigg\}\frac{(i\theta)^\xi}{\xi!} \notag
\end{align}
holds for $\theta \in [-\pi,\pi]$, where the infinite series appearing on both 
sides of (\ref{eq-4-4}) are absolutely convergent for $\theta \in [-\pi,\pi]$.
\end{lemma}

For $p \in \mathbb{N}$, 
it is known that (see, for example, \cite[(4.31),\,(4.32)]{KMTJC})
\begin{align}
& \lim_{L \to \infty}\sum_{-L\leq l\leq L \atop l\not=0} \frac{(-1)^{l}e^{il\theta}}{l^{p}}=2\sum_{j=0}^{p}\ \phi(p-j)\varepsilon_{p-j}\frac{(i\theta)^{j}}{j!}\quad (\theta \in (-\pi,\pi)). \label{eq-4-4-2} 
\end{align}
Note that the left-hand side is uniformly convergent for $\theta \in (-\pi,\pi)$ (see \cite[$\S$ 3.35]{WW}), and is also absolutely convergent for $p\geq 2$. We prove the following lemma. Note that the case when $p$ and $q$ are even has been already proved in \cite[(7.55)]{KMTJC}.

\begin{lemma}\label{Lem-5-1} 
For $p\in \mathbb{N}$, $s \in \mathbb{R}$ with $s>1$ and $x\in \mathbb{C}$ with $|x|= 1$,
\begin{equation}
\begin{split}
& \sum_{l\not=0,\,m\geq 1 \atop l+m\not=0} \frac{(-1)^{l+m}x^m e^{i(l+m)\theta}}{l^{p}m^{s}(l+m)^{q}} \\
& \ \ -2\sum_{j=0}^{p}\ \phi(p-j)\varepsilon_{p-j} \sum_{\xi=0}^{j}\binom{q-1+j-\xi}{q-1}(-1)^{j-\xi}\sum_{m=1}^\infty \frac{(-1)^{m}x^m e^{im\theta}}{m^{s+q+j-\xi}}\frac{(i\theta)^{\xi}}{\xi!} \\
& \ \ +2\sum_{j=0}^{q}\ \phi(q-j)\varepsilon_{q-j} \sum_{\xi=0}^{j}\binom{p-1+j-\xi}{p-1}(-1)^{p-1}\sum_{m=1}^\infty \frac{x^m }{m^{s+p+j-\xi}}\frac{(i\theta)^{\xi}}{\xi!}=0 
\end{split}
\label{eq-5-6}
\end{equation}
for $\theta \in [-\pi,\pi]$. 
\end{lemma}

\begin{proof}
First we assume $p\geq 2$. Then, for $\theta \in (-\pi,\pi)$, it follows from \eqref{eq-4-4-2} that
\begin{align}
& \left( \sum_{l\in \mathbb{Z} \atop l\not=0} \frac{(-1)^{l}e^{il\theta}}{l^{p}}-2\sum_{j=0}^{p}\ \phi(p-j)\varepsilon_{p-j}\frac{(i\theta)^{j}}{j}\right)\sum_{m=1}^\infty \frac{(-1)^{m}x^m e^{im\theta}}{m^s}=0, \label{eq-4-4-3} 
\end{align}
where the left-hand side is absolutely and uniformly convergent for $\theta \in (-\pi,\pi)$. Therefore we have
\begin{equation}
\begin{split}
& \sum_{l\in \mathbb{Z},\ l\not=0 \atop{m\geq 1 \atop l+m\not=0}} \frac{(-1)^{l+m}x^m e^{i(l+m)\theta}}{l^{p}m^{s}}-2\sum_{j=0}^{p}\ \phi(p-j)\varepsilon_{p-j}\left\{ \sum_{m=1}^\infty \frac{(-1)^{m}x^m e^{im\theta}}{m^s}\right\} \frac{(i\theta)^{j}}{j!}\\
& \ \ \ \ =(-1)^{p+1}\sum_{m=1}^\infty \frac{x^m}{m^{s+p}} 
\end{split}
\label{eq-4-4-4} 
\end{equation}
for $\theta \in (-\pi,\pi)$. Now we use Lemma \ref{L-4-2} with $h=1$, $a_1=p$, 
$$C(N)=\sum_{l\not=0,\,m\geq 1 \atop l+m=N} \,\frac{x^m}{l^{p}m^{s}}\ (N \in \mathbb{Z},\,N\not=0)$$
and $D(N;\mu;1)=x^N N^{-s}$ (if $\mu=0$ and $N \geq 1$), or $=0$ (otherwise).
Under these settings, we see that the left-hand side of (\ref{eq-4-4-4}) is of the form (\ref{eq-4-3}). Furthermore the right-hand side of (\ref{eq-4-4-4}) is a constant 
as a function in $\theta$. Therefore we can apply Lemma \ref{L-4-2} with $d=q\in \mathbb{N}$. Then (\ref{eq-4-4}) gives \eqref{eq-5-6} for $p\geq 2$. 

Next we prove the case $p=1$. As we proved above, \eqref{eq-5-6} in the case $p=2$ holds. Replacing $x$ by $-xe^{i\theta}$ in this case, we have
\begin{equation}
\begin{split}
& \sum_{l\not=0,\,m\geq 1 \atop l+m\not=0} \frac{(-1)^{l}x^m e^{il\theta}}{l^{2}m^{s}(l+m)^{q}} \\
& \ \ -2\sum_{j=0}^{2}\ \phi(2-j)\varepsilon_{2-j} \sum_{\xi=0}^{j}\binom{q-1+j-\xi}{q-1}(-1)^{j-\xi}\sum_{m=1}^\infty \frac{x^m}{m^{s+q+j-\xi}}\frac{(i\theta)^{\xi}}{\xi!} \\
& \ \ +2\sum_{j=0}^{q}\ \phi(q-j)\varepsilon_{q-j} \sum_{\xi=0}^{j}\binom{1+j-\xi}{1}(-1)^{1}\sum_{m=1}^\infty \frac{(-1)^mx^m e^{-im\theta}}{m^{s+2+j-\xi}}\frac{(i\theta)^{\xi}}{\xi!}=0 
\end{split}
\label{eq-5-62}
\end{equation}
for $\theta \in [-\pi,\pi]$. We denote the first, the second and the third term on the left hand side of \eqref{eq-5-62} by $I_1(\theta)$, $I_2(\theta)$ and $I_3(\theta)$, respectively. 
We differentiate these terms in $\theta$. We can easily compute $I_1'(\theta)$ and $I_2'(\theta)$. As for $I_3'(\theta)$, we have
\begin{align*}
I_3'(\theta)& = 2\sum_{j=0}^q\phi(q-j)\varepsilon_{q-j}\bigg\{-i\sum_{\xi=0}^{j}(1+j-\xi)(-1)\sum_{m=1}^\infty \frac{(-1)^mx^m e^{-im\theta}}{m^{s+1+j-\xi}}\frac{(i\theta)^{\xi}}{\xi!}\\
& \qquad +i\sum_{\xi=1}^{j}(1+j-\xi)(-1)\sum_{m=1}^\infty \frac{(-1)^mx^m e^{-im\theta}}{m^{s+2+j-\xi}}\frac{(i\theta)^{\xi-1}}{(\xi-1)!}\bigg\}.
\end{align*}
Note that as for the second member in the curly brackets on the right-hand side, $\xi$ may also run from $1$ to $j+1$ because $1+j-(j+1)=0$ in the summand. Hence, by replacing $\xi-1$ by $\xi$, we have 
\begin{align*}
I_3'(\theta)& = 2i\sum_{j=0}^q\phi(q-j)\varepsilon_{q-j}\sum_{\xi=0}^{j}\sum_{m=1}^\infty \frac{(-1)^mx^m e^{-im\theta}}{m^{s+1+j-\xi}}\frac{(i\theta)^{\xi}}{\xi!}.
\end{align*}
Thus, we see that $(I_1'(\theta)+I_2'(\theta)+I_3'(\theta))/i$,
replacing $x$ by $-xe^{i\theta}$, 
gives \eqref{eq-5-6} in the case $p=1$. This completes the proof.
\end{proof}

The following special cases of \eqref{eq-5-6} will be necessary in the next section. 

\begin{corollary} \label{L-MT}
For $p,q \in \mathbb{N}$, $s>1$ and $x,y\in \mathbb{C}$ with $|x|\leq 1$ and $|y|\leq 1$, let 
\begin{equation}
\mathfrak{T}(p,s,q;x,y)=\sum_{l\not=0,\,m\geq 1 \atop l+m\not=0} \frac{x^l y^m}{l^{p}m^{s}(l+m)^{q}}.
\label{def-T}
\end{equation}
Then
\begin{align}
\mathfrak{T}(p,s,q;1,1) & =2(-1)^p \sum_{k=0}^{[p/2]}\ \zeta(2k)\binom{p+q-1-2k}{q-1}\zeta(s+p+q-2k) \label{eq-MT-00}\\
& \ \ +2(-1)^p \sum_{k=0}^{[q/2]}\ \zeta(2k)\binom{p+q-1-2k}{p-1}\zeta(s+p+q-2k), \notag\\
\mathfrak{T}(p,s,q;-1,1) & =2(-1)^p \sum_{k=0}^{[p/2]}\ \phi(2k)\binom{p+q-1-2k}{q-1}\zeta(s+p+q-2k) \label{eq-MT-l}\\
& \ \ +2(-1)^p \sum_{k=0}^{[q/2]}\ \phi(2k)\binom{p+q-1-2k}{p-1}\phi(s+p+q-2k), \notag\\
\mathfrak{T}(p,s,q;1,-1) & =2(-1)^p \sum_{k=0}^{[p/2]}\ \zeta(2k)\binom{p+q-1-2k}{q-1}\phi(s+p+q-2k) \label{eq-MT-m}\\
& \ \ +2(-1)^p \sum_{k=0}^{[q/2]}\ \zeta(2k)\binom{p+q-1-2k}{p-1}\phi(s+p+q-2k), \notag\\
\mathfrak{T}(p,s,q;-1,-1) & =2(-1)^p \sum_{k=0}^{[p/2]}\ \phi(2k)\binom{p+q-1-2k}{q-1}\phi(s+p+q-2k) \label{eq-MT-lm}\\
& \ \ +2(-1)^p \sum_{k=0}^{[q/2]}\ \phi(2k)\binom{p+q-1-2k}{p-1}\zeta(s+p+q-2k). \notag
\end{align}
\end{corollary}

\begin{proof}
We can directly obtain \eqref{eq-MT-lm} and \eqref{eq-MT-l} by letting 
$(x,\theta)=(1,0)$, $(-1,0)$, respectively in \eqref{eq-5-6}.   (Since
$\theta=0$, all the terms corresponding to $\xi\geq 1$ vanish on the right-hand side
of \eqref{eq-5-6}.)
As for \eqref{eq-MT-m} and \eqref{eq-MT-00}, we let 
$(x,\theta)=(-1,\pi)$, $(1,\pi)$, respectively in \eqref{eq-5-6}
and use Lemma \ref{L-4-3}. 
\end{proof}

\begin{remark} \rm 
The results of the above corollary are not new.   The formula
\eqref{eq-MT-00} was first proved in \cite{TsC}, and then, by a different method,
Nakamura \cite[Theorem 3.1]{Naka} has shown all of the above
(see also a survey given in \cite[Section 3]{KMTJC}).
\end{remark}

\ 

\section{The zeta-function of $SU(4)$}\label{sec-4}

Now we start to prove functional relations for zeta-functions for lattices of type $A_3$.
We use the same technique as in our previous paper \cite[Section 7]{KMTJC}. Hence the details of their proofs will be omitted. 

In this section we study the zeta-function associated with the group $SU(4)$.
Our starting point is similar to \eqref{eq-4-4-3}, or 
\cite[Equation\,(7.58)]{KMTJC}.    We begin by considering
\begin{align}
& \left( \sum_{l\in \mathbb{Z} \atop l\not=0} \frac{(-1)^{l}e^{il\theta}}{l^{p}}-2\sum_{j=0}^{p}\ \phi(p-j)\varepsilon_{p-j}\frac{(i\theta)^{j}}{j}\right)\sum_{m\in \mathbb{Z},\,m\not=0 \atop {n\geq 1 \atop m+n\not=0}} \frac{(-1)^{m+n}x^m y^n e^{i(m+n)\theta}}{m^{q}n^s(m+n)^{b}}=0 \label{7-1} 
\end{align}
for $\theta \in [-\pi,\pi]$, where $p,q,b \in \mathbb{N}$ with 
$p\geq 2$, $s \in \mathbb{R}$ with $s>1$ and $x,y\in \mathbb{C}$ with $|x|=|y|=1$.    
Then, by the (almost) same argument 
as in \cite[pp.\,158-160]{KMTJC}, we obtain 
\begin{align}
& \sum_{l,m\not=0,\;n\geq 1\atop {l+m\not=0,\;m+n\not=0 \atop l+m+n\not=0}} \frac{(-1)^{l+m}x^my^n e^{i(l+m)\theta}}{l^{p}m^{q}n^{s}(l+m)^{a}(m+n)^{b}(l+m+n)^{c}} \label{equ-6-1}\\
& \quad =2\sum_{k=0}^{p}\ \phi(p-k)\varepsilon_{p-k}\sum_{\xi=0}^{k}\sum_{\omega=0}^{k-\xi}\binom{\omega+a-1}{\omega}(-1)^\omega \binom{k-\xi-\omega+c-1}{k-\xi-\omega}\notag\\
& \quad \quad \times (-1)^{k-\xi-\omega} \sum_{m\not=0 \atop {n \geq 1 \atop {m+n\not=0}}} \frac{(-1)^{m}x^my^n e^{im\theta}}{m^{q+a+\omega}n^s(m+n)^{b+c+k-\xi-\omega}} \frac{(i\theta)^{\xi}}{\xi!}\notag\\
& \quad -2\sum_{j=0}^{c}\ \phi(c-k)\varepsilon_{c-k}\sum_{\xi=0}^{k}\sum_{\omega=0}^{k-\xi}\binom{\omega+a-1}{\omega}(-1)^\omega \binom{k-\xi-\omega+p-1}{p-1}\notag\\
& \quad \quad \times (-1)^{p-1+a+\omega}\sum_{m\not=0 \atop {n \geq 1 \atop {m+n\not=0}}} \frac{(-1)^{n}x^my^n e^{-in\theta}}{m^{q}n^{s+a+\omega}(m+n)^{p+b+k-\xi-\omega}} \frac{(i\theta)^{\xi}}{\xi!}\notag\\
& \quad -2\sum_{k=0}^{a}\ \phi(a-k)\varepsilon_{a-k}\sum_{\xi=0}^{k}\sum_{\omega=0}^{p-1}\binom{\omega+k-\xi}{\omega}(-1)^\omega\binom{p+c-2-\omega}{p-1-\omega} \notag\\
& \quad \quad \times (-1)^{p-1-\omega} \sum_{m\not=0 \atop {n \geq 1 \atop {m+n\not=0}}} \frac{x^my^n }{m^{q+k-\xi+\omega+1}n^s(m+n)^{p+b+c-1-\omega}} \frac{(i\theta)^{\xi}}{\xi!}\notag\\
& \quad +2\sum_{k=0}^{a}\ \phi(a-k)\varepsilon_{a-k}\sum_{\xi=0}^{k}\sum_{\omega=0}^{c-1}\binom{\omega+k-\xi}{\omega}(-1)^\omega \binom{p+c-2-\omega}{p-1}\notag\\
& \quad \quad \times (-1)^{p+k-\xi+\omega}\sum_{m\not=0 \atop {n \geq 1 \atop {m+n\not=0}}} \frac{x^my^n }{m^{q}n^{s+k-\xi+\omega+1}(m+n)^{p+b+c-1-\omega}} \frac{(i\theta)^{\xi}}{\xi!}\notag
\end{align}
for $\theta \in [-\pi,\pi]$ and $p,q,a,b,c \in \mathbb{N}$. 
(A small difference is that, in the course of the argument, we replaced $x$ by
$-e^{-i\theta}$ in \cite{KMTJC}, while this time we replace $y$ by $-ye^{-i\theta}$.)
Note that \eqref{equ-6-1} in the case $p=1$ can be proved similarly to
Lemma \ref{Lem-5-1}. 

Now we put $(x,y,\theta)=(-1,-1,0)$ in \eqref{equ-6-1}, namely we take notice of the constant term of \eqref{equ-6-1}. 
We proceed similarly to the argument in \cite{KMTJC}; that is, we decompose the
left-hand side of \eqref{equ-6-1} by the method written in \cite[p.\,160]{KMTJC}, 
while apply \eqref{eq-MNOT} to the right-hand side.
Then 
we obtain the following theorem. 

\begin{theorem}\label{P-FR-A_3}
For $p,q,a,b,c\in \mathbb{N}$, 
\begin{align}
& {\zeta}_3((p,q,s,a,b,c),\lambda_2^\vee;SU(4))
+(-1)^p{\zeta}_3((p,a,s,q,c,b),\lambda_2^\vee;SU(4))\label{Fq-SU4}\\
& +(-1)^{p+a}{\zeta}_3((q,a,c,p,s,b),\lambda_2^\vee;SU(4))
+(-1)^{p+a+c}{\zeta}_3((q,s,c,b,a,p),\lambda_2^\vee;SU(4)) \notag\\
& +(-1)^{q}{\zeta}_3((a,q,b,p,s,c),\lambda_2^\vee;SU(4))
+(-1)^{q+b}{\zeta}_3((a,s,b,c,q,p),\lambda_2^\vee;SU(4)) \notag\\
& +(-1)^{q+a}{\zeta}_3((a,p,b,q,c,s),\lambda_2^\vee;SU(4)) 
+(-1)^{q+a+b}{\zeta}_3((a,c,b,s,p,q),\lambda_2^\vee;SU(4))\notag\\
&+(-1)^{q+a+b+c}{\zeta}_3((s,c,p,a,b,q),\lambda_2^\vee;SU(4)) 
+(-1)^{p+q+a}{\zeta}_3((q,p,c,a,b,s),\lambda_2^\vee;SU(4))\notag\\
&+(-1)^{p+q+a+c}{\zeta}_3((q,b,c,s,p,a),\lambda_2^\vee;SU(4))\notag\\
&+(-1)^{p+q+a+b+c}{\zeta}_3((s,b,p,q,c,a),\lambda_2^\vee;SU(4))\notag\\
& \quad =2(-1)^p\sum_{j=0}^{[p/2]}\ \phi(2j)\sum_{\omega=0}^{p-2j}\binom{\omega+a-1}{\omega} \binom{p+c-2j-\omega-1}{c-1}\notag\\
& \qquad \quad \times \mathfrak{T}(q+a+\omega, s,p+b+c-2j-\omega;1,-1) \notag\\
& \quad +2(-1)^{p+a}\sum_{j=0}^{[c/2]}\ \phi(2j)\sum_{\omega=0}^{c-2j}\binom{\omega+a-1}{\omega}\binom{p+c-2j-\omega-1}{p-1}\notag\\
& \qquad \quad \times \mathfrak{T}(q,s+a+\omega,p+b+c-2j-\omega;-1,1)\notag\\
& \quad +2(-1)^p\sum_{j=0}^{[a/2]}\ \phi(2j)\sum_{\omega=0}^{p-1}\binom{\omega+a-2j}{\omega}\binom{p+c-2-\omega}{c-1} \notag\\
& \qquad \quad \times \mathfrak{T}(q+a-2j+\omega+1,s,p+b+c-1-\omega;-1,-1) \notag\\
& \quad +2(-1)^{p+a}\sum_{j=0}^{[a/2]}\ \phi(2j)\sum_{\omega=0}^{c-1}\binom{\omega+a-2j}{\omega}\binom{p+c-2-\omega}{p-1}\notag\\
& \qquad \quad \times \mathfrak{T}(q,s+a-2j+\omega+1,p+b+c-1-\omega;-1,-1)\notag
\end{align}
holds for $s \in \mathbb{C}$ except for singularities of functions on both sides, where $\mathfrak{T}(p,s,q;x,y)$ is defined by \eqref{def-T}.
Moreover, from \eqref{eq-MT-00}-\eqref{eq-MT-lm} we see that the right-hand side of
the above can be written in terms of the Riemann zeta-function. 
\end{theorem}

Setting $(a,b,c,p,q,s)=(2k,2k,2k,2k,2k,2k)$ for $k\in \mathbb{N}$, we obtain 
$${\zeta}_3((2k,2k,2k,2k,2k,2k),\lambda_2^\vee;SU(4))\in \mathbb{Q}\cdot \pi^{12k},$$
and the rational coefficients can be determined explicitly.
This gives an example of \cite[Theorem 3.2]{KMT-Lie}. 

Similarly, by putting $(x,y,\theta)=(1,1,\pi)$ in \eqref{equ-6-1} and using \eqref{eq-MNOT} in Lemma \ref{L-4-3}, we obtain the following theorem for $\zeta_3({\bf s};A_3)=\zeta_3({\bf s},\bs{0};SU(4))$ (see \eqref{zeta-A3-P}). 

\begin{theorem}\label{P-FR-A_3-2}
For $p,q,a,b,c\in \mathbb{N}$, 
\begin{align}
& {\zeta}_3((p,q,s,a,b,c);A_3)
+(-1)^p{\zeta}_3((p,a,s,q,c,b);A_3)\label{Fq-A3}\\
& +(-1)^{p+a}{\zeta}_3((q,a,c,p,s,b);A_3)
+(-1)^{p+a+c}{\zeta}_3((q,s,c,b,a,p);A_3) \notag\\
& +(-1)^{q}{\zeta}_3((a,q,b,p,s,c);A_3)
+(-1)^{q+b}{\zeta}_3((a,s,b,c,q,p);A_3) \notag\\
& +(-1)^{q+a}{\zeta}_3((a,p,b,q,c,s);A_3) 
+(-1)^{q+a+b}{\zeta}_3((a,c,b,s,p,q);A_3)\notag\\
&+(-1)^{q+a+b+c}{\zeta}_3((s,c,p,a,b,q);A_3) 
+(-1)^{p+q+a}{\zeta}_3((q,p,c,a,b,s);A_3)\notag\\
&+(-1)^{p+q+a+c}{\zeta}_3((q,b,c,s,p,a);A_3)+(-1)^{p+q+a+b+c}{\zeta}_3((s,b,p,q,c,a);A_3)\notag\\
& \quad =2(-1)^{p}\sum_{j=0}^{[p/2]}\zeta(2j)\sum_{\omega=0}^{p-2j}\binom{\omega+a-1}{\omega}\binom{p+c-2j-\omega-1}{c-1} \notag\\
& \quad \quad \times \mathfrak{T}({q+a+\omega},s,{p+b+c-2j-\omega};1,1)\notag\\
& \quad +2(-1)^{p+a}\sum_{j=0}^{[c/2]}\zeta(2j)\sum_{\omega=0}^{c-2j}\binom{\omega+a-1}{\omega}\binom{p+c-2j-\omega-1}{p-1} \notag\\
& \quad \quad \times \mathfrak{T}(q,s+a+\omega,{p+b+c-2j-\omega};1,1)\notag\\
& \quad +2(-1)^p \sum_{j=0}^{[a/2]}\zeta(2j)\sum_{\omega=0}^{p-1}\binom{\omega+a-2j}{\omega}\binom{p+c-2-\omega}{c-1} \notag\\
& \quad \quad \times \mathfrak{T}({q+a-2j+\omega+1},s,{p+b+c-1-\omega};1,1)\notag\\
& \quad +2(-1)^{p+a}\sum_{j=0}^{[a/2]}\zeta(2j)\sum_{\omega=0}^{c-1}\binom{\omega+a-2j}{\omega}\binom{p+2-2-\omega}{p-1} \notag\\
& \quad \quad \times \mathfrak{T}({q},{s+a-2j+\omega+1},{p+b+c-\omega-1};1,1)\notag
\end{align}
holds for $s \in \mathbb{C}$ except for singularities of functions on both sides. 
\end{theorem}

When $p,q,a,b,c$ are all even, this theorem has already been proved in 
\cite[Theorem 7.1]{KMTJC}.   The expression of the left-hand side in \cite{KMTJC}
is a little different from the above, but we can easily check that those two
expressions are equal, using \eqref{change-0}.

Setting $(a,b,c,p,q,s)=(2k,2k,2k,2k,2k,2k)$ for $k\in \mathbb{N}$, we obtain 
\eqref{1-2} for $A_3$ with the explicit value of the coefficient. 
On the other hand, when $p=q=a=b=c$ which is an odd integer, the left-hand sides 
of the above two theorems are equal to 0.   This is because
${\zeta}_3((p,p,p,s,p,p);A_3)={\zeta}_3((p,p,p,p,s,p);A_3)$ and
${\zeta}_3((p,p,p,s,p,p),\lambda_2^\vee;SU(4))=
{\zeta}_3((p,p,p,p,s,p),\lambda_2^\vee;SU(4))$,
by \eqref{change-2}, \eqref{change-0}. Hence, unfortunately we can obtain no 
information about, for example, ${\zeta}_3((2k+1,2k+1,2k+1,2k+1,2k+1,2k+1);A_3)$ and ${\zeta}_3(((2k+1,2k+1,2k+1,2k+1,2k+1,2k+1)),\lambda_2^\vee;SU(4))$ $(k\in \mathbb{N}_0)$ from the above theorems. 

However, choosing $(a,b,c,p,q,s)$ suitably, we can obtain some classes of evaluation formulas for them. For example, set 
$$(a,b,c,p,q,s)=(2k+1,2k+1,2k+1,2k+1,2k,2k+1) \qquad (k\in \mathbb{N})$$
in \eqref{Fq-SU4} and \eqref{Fq-A3}. Then 
the left-hand sides of them are  
\begin{align*}
& 2{\zeta}_3((2k+1,2k,2k+1,2k+1,2k+1,2k+1),\lambda_2^\vee;SU(4))\\
& \qquad -2{\zeta}_3((2k+1,2k+1,2k+1,2k,2k+1,2k+1),\lambda_2^\vee;SU(4)) \\
& \ =-2{\zeta}_3((2k,2k+1,2k+1,2k+1,2k+1,2k+1),\lambda_2^\vee;SU(4)),\\
& 2{\zeta}_3((2k+1,2k,2k+1,2k+1,2k+1,2k+1);A_3)\\
& \qquad -2{\zeta}_3((2k+1,2k+1,2k+1,2k,2k+1,2k+1);A_3) \\
& \ =-2{\zeta}_3((2k,2k+1,2k+1,2k+1,2k+1,2k+1);A_3),
\end{align*}
respectively, by using the relation
$$\frac{1}{l^{2k+1}m^{2k}(l+m)^{2k+1}}-\frac{1}{l^{2k+1}m^{2k+1}(l+m)^{2k}}=-\frac{1}{l^{2k}m^{2k+1}(l+m)^{2k+1}}.$$
Therefore we obtain the following.

\begin{proposition}\label{P-new}
For $k\in \mathbb{N}$, 
\begin{align*}
& {\zeta}_3((2k,2k+1,2k+1,2k+1,2k+1,2k+1),\lambda_2^\vee;SU(4)) \in \mathbb{Q}[\{\zeta(j)\,|\,j\in \mathbb{N}_{>1}\}],\\
& {\zeta}_3((2k,2k+1,2k+1,2k+1,2k+1,2k+1);A_3)\in \mathbb{Q}[\{\zeta(j)\,|\,j\in \mathbb{N}_{>1}\}],
\end{align*}
and the rational coefficients can be determined explicitly.
\end{proposition}

\begin{example} \label{Exam-A3} \rm 
Setting $(a,b,c,p,q,s)=(2k+1,2k+1,2k+1,2k+1,2k,2k+1)$ in \eqref{Fq-A3}, we can obtain
\begin{align*}
& \,\zeta_3(2,3,3,3,3,3;A_3) =\frac{\pi^{6}}{63}\,\zeta(11) +\frac{199\pi^{4}}{30}\,\zeta(13) -{365}\pi^{2}\,\zeta(15) +{2941}\,\zeta(17),\\
& \,\zeta_3(4,5,5,5,5,5;A_3)\\
& \quad =\frac{152\pi^{12}}{18243225}\,\zeta(17) +\frac{17\pi^{10}}{6237}\,\zeta(19) +\frac{29\pi^{8}}{54}\,\zeta(21) +\frac{979\pi^{6}}{9}\,\zeta(23) \\
& \quad +\frac{15585\pi^{4}}{2}\,\zeta(25) -{660975\pi^{2}}\,\zeta(27) +{5654565}\,\zeta(29),\\
& \,\zeta_3(6,7,7,7,7,7;A_3)\\
& \quad =\frac{2062\pi^{18}}{506224616625}\,\zeta(23) +\frac{11776\pi^{16}}{5367718125}\,\zeta(25) +\frac{8\pi^{14}}{13365}\,\zeta(27) \\
& \quad +\frac{10223594\pi^{12}}{91216125}\,\zeta(29) +\frac{103486\pi^{10}}{6075}\,\zeta(31) +\frac{5459978\pi^{8}}{2025}\,\zeta(33) \\
& \quad +\frac{3464974\pi^{6}}{15}\,\zeta(35) +\frac{41963621\pi^{4}}{3}\,\zeta(37) -{1456076440\pi^{2}}\,\zeta(39) \\
& \quad +{12758984832}\,\zeta(41),\\
& \,\zeta_3(8,9,9,9,9,9;A_3)\\
& \ =\frac{64586\pi^{24}}{37355158168453125}\,\zeta(29) +\frac{422704\pi^{22}}{298841265347625}\,\zeta(31) \\
& \quad +\frac{10664\pi^{20}}{19088409375}\,\zeta(33) +\frac{663259\pi^{18}}{4632120675}\,\zeta(35) +\frac{2307883\pi^{16}}{84341250}\,\zeta(37)  \\
& \quad +\frac{6327646\pi^{14}}{1488375}\,\zeta(39) +\frac{860790601\pi^{12}}{1488375}\,\zeta(41)+\frac{380997529\pi^{10}}{4725}\,\zeta(43) \\
& \quad +\frac{7867619353\pi^{8}}{1050}\,\zeta(45) +\frac{164035120733\pi^{6}}{315}\,\zeta(47) +\frac{59740238129\pi^{4}}{2}\,\zeta(49) \\
& \quad -{3514635376395}\pi^{2}\,\zeta(51) +{31198575194215}\,\zeta(53),\\
& \,\zeta_3(10,11,11,11,11,11,11;A_3)\\
& \quad =\frac{221912776\pi^{30}}{332660210652234981140625}\,\zeta(35) +\frac{10705232\pi^{28}}{13854831558583640625}\,\zeta(37) \\
& \quad +\frac{5135896\pi^{26}}{12250072111921875}\,\zeta(39) +\frac{4767865562\pi^{24}}{33250195732359375}\,\zeta(41) +\frac{222974564\pi^{22}}{6269397175125}\,\zeta(43) \\
& \quad +\frac{24806393774\pi^{20}}{3569532553125}\,\zeta(45)+\frac{2589565814\pi^{18}}{2290609125}\,\zeta(47) +\frac{1188339011\pi^{16}}{7441875}\,\zeta(49) \\
& \quad +\frac{3650193872\pi^{14}}{178605}\,\zeta(51) +\frac{11782765221344\pi^{12}}{4465125}\,\zeta(53) +\frac{35232949154\pi^{10}}{135}\,\zeta(55)  \\
& \quad +\frac{18601660627979\pi^{8}}{945}\,\zeta(57)+\frac{393366314952754\pi^{6}}{315}\,\zeta(59) \\
& \quad +\frac{1050680447134747\pi^{4}}{15}\,\zeta(61) -{8947964548486678}\pi^{2}\,\zeta(63) \\
& \quad +{80075393000830422}\,\zeta(65).
\end{align*}
Also, setting $(a,b,c,p,q,s)=(2k+1,2k+1,2k+1,2k+1,2k,2k+1)$ in \eqref{Fq-SU4}, we can obtain
\begin{align*}
& {\zeta}_3((2,3,3,3,3,3),\lambda_2^\vee;SU(4))\\
& \ =\frac{17\pi^8}{344064}\zeta(9) + \frac{22847\pi^6}{1720320}\zeta(11) + \frac{49005\pi^4}{16384}\zeta(13) + \frac{3768307\pi^2}{98304}\zeta(15) - \frac{11189819}{16384}\zeta(17),\\
& {\zeta}_3((4,5,5,5,5,5),\lambda_2^\vee;SU(4))\\
& \ =\frac{693547\pi^{14}}{51011754393600}\zeta(15) + \frac{714624223\pi^{12}}{81618807029760}\zeta(17) + \frac{28726157\pi^{10}}{11072962560}\zeta(19) \\
& \ + \frac{25906094783\pi^8}{54358179840}\zeta(21) + \frac{9177921545\pi^6}{113246208}\zeta(23) + \frac{2422120970909\pi^4}{671088640}\zeta(25) \\
& \ + \frac{7798050014825\pi^2}{134217728}\zeta(27) - \frac{270498379148235}{268435456}\zeta(29),\\
& {\zeta}_3((6,7,7,7,7,7),\lambda_2^\vee;SU(4))\\
& \ = \frac{2752145869\pi^{20}}{773055350341160140800}\zeta(21) + \frac{1098434242057681\pi^{18}}{255108265612582846464000}\zeta(23) \\
& \ + \frac{150866953637\pi^{16}}{68882685493248000}\zeta(25) + \frac{20612241204619\pi^{14}}{34824024332697600}\zeta(27) \\
& \ + \frac{19614225808011463\pi^{12}}{179094982282444800}\zeta(29) + \frac{6776217678200971\pi^{10}}{417470821171200}\zeta(31) \\
& \ + \frac{337234670875566533\pi^8}{139156940390400}\zeta(33) + \frac{7289362333395816433\pi^6}{43293270343680}\zeta(35) \\
& \ + \frac{3412540143011100899\pi^4}{515396075520}\zeta(37) + \frac{32450037853433343325\pi^2}{274877906944}\zeta(39) \\
& \ - \frac{274409134558621990125}{137438953472}\zeta(41).
\end{align*}
The authors also checked, by using Mathematica $8$, that the above formulas agree
with numerical computation, based on the definitions of zeta-functions.
\end{example}

\section{The zeta-function of $SO(6)$}\label{sec-5}

Next we consider $\zeta_3({\bf s};SO(6))$. 
It follows from \eqref{zeta-A3-P} and \eqref{zeta-A3-L} that 
\begin{equation}
\zeta_3({\bf s};SO(6))=\frac{1}{2} \left\{\zeta_3({\bf s};A_3)+{\zeta}_3({\bf s},\lambda_2^\vee;SU(4))\right\}. \label{rel-P-L}
\end{equation}
Combining Theorems \ref{P-FR-A_3} and \ref{P-FR-A_3-2} and using \eqref{rel-P-L}, we can obtain functional relations among $\zeta_3({\bf s};SO(6))$ and $\zeta(s)$.

\begin{theorem}\label{T-SO6}
For $p,q,a,b,c\in \mathbb{N}$, 
\begin{align}
& {\zeta}_3((p,q,s,a,b,c);SO(6))
+(-1)^p{\zeta}_3((p,a,s,q,c,b);SO(6))\label{Fq-SO6}\\
& +(-1)^{p+a}{\zeta}_3((q,a,c,p,s,b);SO(6))
+(-1)^{p+a+c}{\zeta}_3((q,s,c,b,a,p);SO(6)) \notag\\
& +(-1)^{q}{\zeta}_3((a,q,b,p,s,c);SO(6))
+(-1)^{q+b}{\zeta}_3((a,s,b,c,q,p);SO(6)) \notag\\
& +(-1)^{q+a}{\zeta}_3((a,p,b,q,c,s);SO(6)) 
+(-1)^{q+a+b}{\zeta}_3((a,c,b,s,p,q);SO(6))\notag\\
&+(-1)^{q+a+b+c}{\zeta}_3((s,c,p,a,b,q);SO(6)) 
+(-1)^{p+q+a}{\zeta}_3((q,p,c,a,b,s);SO(6))\notag\\
&+(-1)^{p+q+a+c}{\zeta}_3((q,b,c,s,p,a);SO(6))+(-1)^{p+q+a+b+c}{\zeta}_3((s,b,p,q,c,a);SO(6))\notag\\
& =\frac{1}{2}\left(J_0+J_2\right)\notag 
\end{align}
holds for $s \in \mathbb{C}$ except for singularities of functions on both sides, where $J_0$ and $J_2$ are the right-hand sides of \eqref{Fq-A3} and \eqref{Fq-SU4}, respectively.
\end{theorem}

Similarly to Proposition \ref{P-new} and Example \ref{Exam-A3}, we can obtain the following.

\begin{proposition}\label{P-SO6}
For $k\in \mathbb{N}$, 
\begin{align*}
& \zeta_3((2k,2k,2k,2k,2k,2k);SO(6))\in \mathbb{Q}\cdot \pi^{12k},\\
& \zeta_3((2k,2k+1,2k+1,2k+1,2k+1,2k+1);SO(6))\in \mathbb{Q}[\{\zeta(j)\,|\,j\in \mathbb{N}_{>1}\}],
\end{align*}
and the rational coefficients can be determined explicitly.
\end{proposition}

\begin{example}\label{Exam-SO6} \rm 
Setting $(a,b,c,p,q,s)=(2,2,2,2,2,s)$ in \eqref{Fq-SO6}, we obtain
\begin{align}
& 2\zeta_3((2,s,2,2,2,2);SO(6)) +4\zeta_3((2,2,s,2,2,2);SO(6)) \label{A3-L-FR}\\
& +4\zeta_3((2,2,2,s,2,2);SO(6)) +2\zeta_3((2,2,2,2,2,s);SO(6)) \notag\\
&=\left(93\cdot 2^{-s-8}+306 \right)\zeta(s+10) +\left( 3\cdot 2^{-s-4}-260\right)\zeta(2)\zeta(s+8)\notag\\
&-\left(67\cdot 2^{-s-6}-110\right)\zeta(4)\zeta(s+6)-\frac{1}{8}\left(5\cdot 2^{-s-3}-21 \right)\zeta(6)\zeta(s+4). \notag
\end{align}
In particular, when $s=2$ in \eqref{A3-L-FR}, we have
\begin{equation}
\zeta_3((2,2,2,2,2,2);SO(6))=\frac{10411}{1307674368000}\pi^{12}.
\label{A3-L-val1}
\end{equation}
Also, combining \eqref{rel-P-L} and the results in Example \ref{Exam-A3}, we obtain, for example, 
\begin{align}
\zeta_3((2,3,3,3,3,3);SO(6)) =& \frac{17\pi^8}{688128}\zeta(9) + \frac{150461\pi^6}{10321920}\zeta(11) + \frac{2365283\pi^4}{491520}\zeta(13) \label{SO6-val}\\
& \ - \frac{32112653\pi^2}{196608}\zeta(15) + \frac{36995525}{32768}\zeta(17). \notag
\end{align}

\end{example}

\ 

\section{The zeta-function of $PU(4)$}\label{sec-6}

Finally we consider the case of the group $PU(4)$. 
An interesting feature in this case is the appearance of a Dirichlet $L$-function,
so we will describe some details of the argument.

First we slightly generalize the results used in the previous sections. 
Let 
$$\phi(s;\alpha)=\sum_{m=1}^{\infty}e^{2m\pi i\alpha} m^{-s}$$
 be the Lerch zeta-function for $\alpha\in \mathbb{R}$. We can easily see that $\phi(s;1/2)$ is equal to $\phi(s)=\left(2^{1-s}-1\right)\zeta(s)$ used in Section \ref{sec-3}, and 
\begin{align}
& \phi(s;1/4)=2^{-s}\left(2^{1-s}-1\right)\zeta(s)+iL(s,\chi_4), \label{pl+i}\\
& \phi(s;-1/4)=2^{-s}\left(2^{1-s}-1\right)\zeta(s)-iL(s,\chi_4), \label{pl-i}
\end{align}
where $L(s,\chi_4)=\sum_{m\geq 0}(-1)^m (2m+1)^{-s}$ be the Dirichlet $L$-function associated with the primitive Dirichlet character $\chi_4$ of conductor $4$.
Moreover we let 
\begin{align*}
& \varLambda(s;i) =\sum_{m\in \mathbb{Z}\smallsetminus \{ 0\}}\frac{i^m}{m^s}=2^{-s}\left(1+e^{-\pi is}\right)\left(2^{1-s}-1\right)\zeta(s)+i\left(1-e^{-\pi is}\right)L(s,\chi_4),\\
& \varLambda(s;-i) =\sum_{m\in \mathbb{Z}\smallsetminus \{ 0\}}\frac{(-i)^m}{m^s}=2^{-s}\left(1+e^{-\pi is}\right)\left(2^{1-s}-1\right)\zeta(s)-i\left(1-e^{-\pi is}\right)L(s,\chi_4),
\end{align*}
where $m^{-s}=\exp(-s(\log|m|+\pi i))$ for $m<0$.
In particular, for $k\in \mathbb{N}$ and $l\in \mathbb{N}_0$, we have
\begin{align}
& \varLambda(2k;i)=2^{1-2k}\left(2^{1-2k}-1\right)\zeta(2k);\ \ \varLambda(2l+1;i)=2iL(2l+1,\chi_4), \label{ee-6-2}\\
& \varLambda(2k;-i)=2^{1-2k}\left(2^{1-2k}-1\right)\zeta(2k);\ \ \varLambda(2l+1;-i)=-2iL(2l+1,\chi_4). \label{ee-6-3}
\end{align}
Also, it is well-known that 
\begin{equation}
\lim_{K\to \infty}\sum_{k=-K \atop k\not=0}^{K}\frac{e^{2\pi i k\alpha}}{k^j}=-B_j\left(\alpha\right)\frac{(2\pi i)^j}{j!}\ \ (j\in \mathbb{N};\,\alpha\in [0,1)) \label{Ber-Lerch}
\end{equation}
(see, for example, \cite[Theorem 12.19]{Ap}). 
Here, setting $c=\pi/2$ and $3\pi/2$ in \eqref{eq-5-6-2} and using \eqref{Ber-Lerch} with $\alpha=\pm 1/4$, we obtain the following.

\begin{lemma} \label{L-6-1}
For any $p \in \mathbb{N}$ and any function $h:\mathbb{N}_{0} \to \mathbb{C}$, 
\begin{align}
& \sum_{j=0}^{p} \phi(p-j) \varepsilon_{p-j}\sum_{\xi=0}^{j} h(j-\xi)\frac{(-i\pi/2)^{\xi}}{\xi!} =\frac{1}{2}\sum_{\xi=0}^{p}\varLambda(\xi;i)h(p-\xi), \label{ee-6-4}\\
& \sum_{j=0}^{p} \phi(p-j) \varepsilon_{p-j}\sum_{\xi=0}^{j} h(j-\xi)\frac{(i\pi/2)^{\xi}}{\xi!} =\frac{1}{2}\sum_{\xi=0}^{p}\varLambda(\xi;-i)h(p-\xi). \label{ee-6-5}
\end{align}
\end{lemma}

Set $(x,\theta)=(-i,0),\,(-i,\pi/2),\,(-1,\pi/2),\,(i,0),\,(i,-\pi/2),\,(-1,-\pi/2)$ in \eqref{eq-5-6} and use Lemma \ref{L-6-1}. Then, by the same method as in the proof of Corollary \ref{L-MT}, we obtain the following. 

\begin{lemma} \label{L-6-2}
For $p,q \in \mathbb{N}$ and $s>1$, 
\begin{align}
\mathfrak{T}(p,s,q;-1,i) & =2(-1)^p \sum_{k=0}^{[p/2]}\ \phi(2k)\binom{p+q-1-2k}{q-1}\phi(s+p+q-2k;1/4) \label{eq-MT-1+i}\\
& \ \ +2(-1)^p \sum_{k=0}^{[q/2]}\ \phi(2k)\binom{p+q-1-2k}{p-1}\phi(s+p+q-2k;-1/4), \notag\\
\mathfrak{T}(p,s,q;-i,-1) & =(-1)^p \sum_{l=0}^{p}\ \varLambda(l;-i)(-1)^l \binom{p+q-1-l}{q-1}\phi(s+p+q-l) \label{eq-MT-i-1}\\
& \ \ +(-1)^p \sum_{l=0}^{q}\ \varLambda(l;-i)\binom{p+q-1-l}{p-1}\phi(s+p+q-l;-1/4), \notag\\
\mathfrak{T}(p,s,q;-i,i) & =(-1)^p \sum_{l=0}^{p}\ \varLambda(l;-i)(-1)^l \binom{p+q-1-l}{q-1}\phi(s+p+q-l;1/4) \label{eq-MT-i+i}\\
& \ \ +(-1)^p \sum_{l=0}^{q}\ \varLambda(l;-i)\binom{p+q-1-l}{p-1}\phi(s+p+q-l), \notag\\
\mathfrak{T}(p,s,q;-1,-i) & =2(-1)^p \sum_{k=0}^{[p/2]}\ \phi(2k)\binom{p+q-1-2k}{q-1}\phi(s+q+j;-1/4) \label{eq-MT-1-i}\\
& \ \ +2(-1)^p \sum_{k=0}^{[q/2]}\ \phi(2k)\binom{p+q-1-2k}{p-1}\phi(s+p+q-2k;1/4), \notag\\
\mathfrak{T}(p,s,q;i,-1) & =(-1)^p \sum_{l=0}^{p}\ \varLambda(l;i)(-1)^l \binom{p+q-1-l}{q-1}\phi(s+p+q-l) \label{eq-MT+i-1}\\
& \ \ +(-1)^p \sum_{l=0}^{q}\ \varLambda(l;i)\binom{p+q-1-l}{p-1}\phi(s+p+q-l;1/4), \notag\\
\mathfrak{T}(p,s,q;i,-i) & =(-1)^p \sum_{l=0}^{p}\ \varLambda(l;i)(-1)^l \binom{p+q-1-l}{q-1}\phi(s+p+q-l;-1/4) \label{eq-MT-i+i}\\
& \ \ +(-1)^p \sum_{l=0}^{q}\ \varLambda(l;i)\binom{p+q-1-l}{p-1}\phi(s+p+q-l). \notag
\end{align}
\end{lemma}

Setting $(x,y,\theta)=(-i,i,\pi/2)$ and $(i,-i,-\pi/2)$ in \eqref{equ-6-1}, and using Lemma \ref{L-6-1}, we obtain the following.

\begin{theorem}\label{T-6-3} 
For $p,q,a,b,c\in \mathbb{N}$, 
\begin{align}
& {\zeta}_3((p,q,s,a,b,c),\lambda_1^\vee;SU(4))
+(-1)^p{\zeta}_3((p,a,s,q,c,b),\lambda_1^\vee;SU(4))\label{Fq-lambda1}\\
& +(-1)^{p+a}{\zeta}_3((q,a,c,p,s,b),\lambda_1^\vee;SU(4))
+(-1)^{p+a+c}{\zeta}_3((q,s,c,b,a,p),\lambda_1^\vee;SU(4)) \notag\\
& +(-1)^{q}{\zeta}_3((a,q,b,p,s,c),\lambda_1^\vee;SU(4))
+(-1)^{q+b}{\zeta}_3((a,s,b,c,q,p),\lambda_1^\vee;SU(4)) \notag\\
& +(-1)^{q+a}{\zeta}_3((a,p,b,q,c,s),\lambda_1^\vee;SU(4)) 
+(-1)^{q+a+b}{\zeta}_3((a,c,b,s,p,q),\lambda_1^\vee;SU(4))\notag\\
&+(-1)^{q+a+b+c}{\zeta}_3((s,c,p,a,b,q),\lambda_1^\vee;SU(4)) 
+(-1)^{p+q+a}{\zeta}_3((q,p,c,a,b,s),\lambda_1^\vee;SU(4))\notag\\
&+(-1)^{p+q+a+c}{\zeta}_3((q,b,c,s,p,a),\lambda_1^\vee;SU(4))\notag\\
&+(-1)^{p+q+a+b+c}{\zeta}_3((s,b,p,q,c,a),\lambda_1^\vee;SU(4))\notag\\
& \quad =(-1)^p \sum_{j=0}^{p}\ \varLambda(j;-i)(-1)^j \sum_{\omega=0}^{p-j}\binom{\omega+a-1}{\omega} \binom{p+c-j-\omega-1}{c-1}\notag\\
& \qquad \quad \times \mathfrak{T}(q+a+\omega, s,p+b+c-j-\omega;-1,i) \notag\\
& \quad +(-1)^{p+a}\sum_{j=0}^{c}\ \varLambda(j;-i)\sum_{\omega=0}^{c-j}\binom{\omega+a-1}{\omega}\binom{p+c-j-\omega-1}{p-1}\notag\\
& \qquad \quad \times \mathfrak{T}(q,s+a+\omega,p+b+c-j-\omega;-i,-1)\notag\\
& \quad +(-1)^p\sum_{j=0}^{a}\ \varLambda(j;-i)\sum_{\omega=0}^{p-1}\binom{\omega+a-j}{\omega}\binom{p+c-2-\omega}{c-1} \notag\\
& \qquad \quad \times \mathfrak{T}(q+a-j+\omega+1,s,p+b+c-1-\omega;-i,i) \notag\\
& \quad +(-1)^{p+a}\sum_{j=0}^{a}\ \varLambda(j;-i)(-1)^j \sum_{\omega=0}^{c-1}\binom{\omega+a-j}{\omega}\binom{p+c-2-\omega}{p-1}\notag\\
& \qquad \quad \times \mathfrak{T}(q,s+a-j+\omega+1,p+b+c-1-\omega;-i,i)\notag
\end{align}
and 
\begin{align}
& {\zeta}_3((p,q,s,a,b,c),\lambda_3^\vee;SU(4))
+(-1)^p{\zeta}_3((p,a,s,q,c,b),\lambda_3^\vee;SU(4))\label{Fq-lambda3}\\
& +(-1)^{p+a}{\zeta}_3((q,a,c,p,s,b),\lambda_3^\vee;SU(4))
+(-1)^{p+a+c}{\zeta}_3((q,s,c,b,a,p),\lambda_3^\vee;SU(4)) \notag\\
& +(-1)^{q}{\zeta}_3((a,q,b,p,s,c),\lambda_3^\vee;SU(4))
+(-1)^{q+b}{\zeta}_3((a,s,b,c,q,p),\lambda_3^\vee;SU(4)) \notag\\
& +(-1)^{q+a}{\zeta}_3((a,p,b,q,c,s),\lambda_3^\vee;SU(4)) 
+(-1)^{q+a+b}{\zeta}_3((a,c,b,s,p,q),\lambda_3^\vee;SU(4))\notag\\
&+(-1)^{q+a+b+c}{\zeta}_3((s,c,p,a,b,q),\lambda_3^\vee;SU(4)) 
+(-1)^{p+q+a}{\zeta}_3((q,p,c,a,b,s),\lambda_3^\vee;SU(4))\notag\\
&+(-1)^{p+q+a+c}{\zeta}_3((q,b,c,s,p,a),\lambda_3^\vee;SU(4))\notag\\
&+(-1)^{p+q+a+b+c}{\zeta}_3((s,b,p,q,c,a),\lambda_3^\vee;SU(4))\notag\\
& \quad =(-1)^p \sum_{j=0}^{p}\ \varLambda(j;i)(-1)^j \sum_{\omega=0}^{p-j}\binom{\omega+a-1}{\omega} \binom{p+c-j-\omega-1}{c-1}\notag\\
& \qquad \quad \times \mathfrak{T}(q+a+\omega, s,p+b+c-j-\omega;-1,-i) \notag\\
& \quad +(-1)^{p+a}\sum_{j=0}^{c}\ \varLambda(j;i)\sum_{\omega=0}^{c-j}\binom{\omega+a-1}{\omega}\binom{p+c-j-\omega-1}{p-1}\notag\\
& \qquad \quad \times \mathfrak{T}(q,s+a+\omega,p+b+c-j-\omega;i,-1)\notag\\
& \quad +(-1)^p\sum_{j=0}^{a}\ \varLambda(j;i)\sum_{\omega=0}^{p-1}\binom{\omega+a-j}{\omega}\binom{p+c-2-\omega}{c-1} \notag\\
& \qquad \quad \times \mathfrak{T}(q+a-j+\omega+1,s,p+b+c-1-\omega;i,-i) \notag\\
& \quad +(-1)^{p+a}\sum_{j=0}^{a}\ \varLambda(j;i)(-1)^j \sum_{\omega=0}^{c-1}\binom{\omega+a-j}{\omega}\binom{p+c-2-\omega}{p-1}\notag\\
& \qquad \quad \times \mathfrak{T}(q,s+a-j+\omega+1,p+b+c-1-\omega;i,-i)\notag
\end{align}
hold for $s \in \mathbb{C}$ except for singularities of functions on both sides.
Moreover, since $\varLambda(j;\pm i)$ and $\mathfrak{T}(p,s,q;x,y)$ satisfy 
\eqref{ee-6-2}-\eqref{ee-6-3} and \eqref{eq-MT-1+i}-\eqref{eq-MT-i+i}, respectively,
we find that the right-hand sides of \eqref{Fq-lambda1} and \eqref{Fq-lambda3}
can be written in terms of $\zeta(s)$ and $L(s,\chi_4)$.
\end{theorem}

Setting $(a,b,c,p,q,s)=(2k,2k,2k,2k,2k,2k)$ for $k\in \mathbb{N}$, we obtain 
$${\zeta}_3((2k,2k,2k,2k,2k,2k),\lambda_j^\vee;SU(4))\in \mathbb{Q}\cdot \pi^{12k}\quad (j=1,3),$$
because, since $\chi_4$ is an odd character, 
$L(2l+1,\chi_4)\in\mathbb{Q}\cdot\pi^{2l+1}$
(see, for example, \cite[p.12]{Iwa}).
This is again an example of \cite[Theorem 3.2]{KMT-Lie}. 

Now we note that
\begin{align}
1-i^{3l+2m+n}+(-1)^{l+n}-i^{l+2m+3n}=\left\{
\begin{array}{ll}
  4 & {\rm if}\; l+2m+3n\equiv 2 \; {\rm (mod\; 4)},\\
  0 & {\rm otherwise.}
\end{array}
\right.
\end{align}
This is because $(-1)^{l+n}=1$ and $i^{3l+2m+n}=i^{l+2m+3n}$ when $l$ and $n$ are
both even or both odd, while $(-1)^{l+n}=-1$ and $i^{3l+2m+n}=-i^{l+2m+3n}$  
otherwise.   Therefore
\begin{align}
& \zeta_3({\bf s},\{{\bf 0}\};PU(4)) \label{relation}\\
& =\frac{1}{4}\bigg( \zeta_3({\bf s},A_3)-\zeta_3({\bf s},\lambda_1^\vee;SU(4))+\zeta_3({\bf s},\lambda_2^\vee;SU(4))-\zeta_3({\bf s},\lambda_3^\vee;SU(4))\bigg),\notag
\end{align}
which is further equal to
\begin{align*}
 \frac{1}{4}\bigg( 2\zeta_3({\bf s},SO(6))-\zeta_3({\bf s},\lambda_1^\vee;SU(4))-\zeta_3({\bf s},\lambda_3^\vee;SU(4))\bigg)
\end{align*}
by \eqref{rel-P-L}. Hence it follows from \eqref{change-1} and \eqref{change-2} that 
\begin{equation}
\zeta_3((s_1,s_2,s_3,s_4,s_5,s_6),\{{\bf 0}\};PU(4))=\zeta_3((s_3,s_2,s_1,s_5,s_4,s_6),\{{\bf 0}\};PU(4)). \label{relation-2}
\end{equation}
Using \eqref{relation} and combining Theorem \ref{T-SO6} and Theorem \ref{T-6-3}, we obtain the following functional relation 
among $\zeta_3({\bf s},\{{\bf 0}\};PU(4))$, $\zeta(s)$ and $L(s,\chi_4)$. 

\begin{theorem}\label{T-6-4} 
For $p,q,a,b,c\in \mathbb{N}$, 
\begin{align}
& {\zeta}_3((p,q,s,a,b,c),{\bf 0};PU(4))
+(-1)^p{\zeta}_3((p,a,s,q,c,b),{\bf 0};PU(4))\label{Fq-PU4}\\
& +(-1)^{p+a}{\zeta}_3((q,a,c,p,s,b),{\bf 0};PU(4))
+(-1)^{p+a+c}{\zeta}_3((q,s,c,b,a,p),{\bf 0};PU(4)) \notag\\
& +(-1)^{q}{\zeta}_3((a,q,b,p,s,c),{\bf 0};PU(4))
+(-1)^{q+b}{\zeta}_3((a,s,b,c,q,p),{\bf 0};PU(4)) \notag\\
& +(-1)^{q+a}{\zeta}_3((a,p,b,q,c,s),{\bf 0};PU(4)) 
+(-1)^{q+a+b}{\zeta}_3((a,c,b,s,p,q),{\bf 0};PU(4))\notag\\
&+(-1)^{q+a+b+c}{\zeta}_3((s,c,p,a,b,q),{\bf 0};PU(4)) 
+(-1)^{p+q+a}{\zeta}_3((q,p,c,a,b,s),{\bf 0};PU(4))\notag\\
&+(-1)^{p+q+a+c}{\zeta}_3((q,b,c,s,p,a),{\bf 0};PU(4))\notag\\
&+(-1)^{p+q+a+b+c}{\zeta}_3((s,b,p,q,c,a),{\bf 0};PU(4))\notag\\
& = \frac{1}{4}\left(J_0-J_1+J_2-J_3\right),\notag 
\end{align}
where $J_0,J_1,J_2,J_3$ are the right-hand sides of \eqref{Fq-A3}, \eqref{Fq-lambda1}, \eqref{Fq-SU4}, \eqref{Fq-lambda3}, respectively.
\end{theorem}

Setting $(a,b,c,p,q,s)=(2k,2k,2k,2k,2k,2k)$ for $k\in \mathbb{N}$, we obtain 
$${\zeta}_3((2k,2k,2k,2k,2k,2k),{\bf 0};PU(4))\in \mathbb{Q}\cdot \pi^{12k}.$$
Also, set $(a,b,c,p,q,s)=(2k+1,2k+1,2k+1,2k+1,2k,2k+1)$ in \eqref{Fq-lambda1}, \eqref{Fq-lambda3} and \eqref{Fq-PU4}. 
Then, by the same method as in the proof of Proposition \ref{P-new}, we obtain the following. 

\begin{proposition}\label{P-new-2}
For $k\in \mathbb{N}$ and $j=1,3$, 
\begin{align*}
& {\zeta}_3((2k,2k+1,2k+1,2k+1,2k+1,2k+1),\lambda_j^\vee;SU(4)) \\
& \qquad \in \mathbb{Q}[\{\zeta(j+1),\,L(j,\chi_4)\,|\,j\in \mathbb{N}\}],
\end{align*}
and so
\begin{align*}
& {\zeta}_3((2k,2k+1,2k+1,2k+1,2k+1,2k+1),\{{\bf 0}\};PU(4)) \\
& \qquad \in \mathbb{Q}[\{\zeta(j+1),\,L(j;\chi_4)\,|\,j\in \mathbb{N}\}],
\end{align*}
where the rational coefficients can be determined explicitly.
\end{proposition}

\begin{example} \label{Ex-6-1} \rm 
Setting $(a,b,c,p,q,s)=(2,2,2,2,2,s)$ in \eqref{Fq-lambda1} and \eqref{Fq-lambda3}, 
and using the data 
$L(1,\chi_4) = \pi/4,\;L(3,\chi_4) = \pi^3/32,\;L(5,\chi_4) = (5/1536)\pi^5$,
we obtain
\begin{align*}
& 2\bigg\{\zeta_3((s,2,2,2,2,2),\lambda_1^\vee;SU(4))+\zeta_3((2,s,2,2,2,2),\lambda_1^\vee;SU(4)) \\
& \quad +\zeta_3((2,2,s,2,2,2),\lambda_1^\vee;SU(4)) +\zeta_3((2,2,2,s,2,2),\lambda_1^\vee;SU(4)) \notag\\
& \quad +\zeta_3((2,2,2,2,s,2),\lambda_1^\vee;SU(4)) +\zeta_3((2,2,2,2,2,s),\lambda_1^\vee;SU(4))\bigg\} \\
& = \left(372\cdot 2^{-s-10} + 306\right)\left(2^{-s-9}-1\right)
\zeta(s+10)+100\pi L(s+9,\chi_4)  \notag\\
& \quad +\left(7\cdot 2^{-s-8} + \frac{32}{3}\right)\left(2^{-s-7}-1\right)\pi^2\zeta(s+8)+ \frac{17}{6}\pi^3 L(s+7,\chi_4)   \notag\\
& \quad + \left(\frac{113\cdot 2^{-s-6}}{1440}+ \frac{1}{288}\right)\left(2^{-s-5}-1\right)\pi^4 \zeta(s+6) \notag\\
& \quad + \frac{1}{32}\pi^5 L(s+5,\chi_4)+\frac{289\cdot 2^{-s-4}}{241920}\left(2^{-s-3}-1\right)\pi^6 \zeta(s+4), \notag 
\end{align*}
and 
\begin{align*}
& 2\bigg\{\zeta_3((s,2,2,2,2,2),\lambda_3^\vee;SU(4))+\zeta_3((2,s,2,2,2,2),\lambda_3^\vee;SU(4)) \\
& \quad +\zeta_3((2,2,s,2,2,2),\lambda_3^\vee;SU(4)) +\zeta_3((2,2,2,s,2,2),\lambda_3^\vee;SU(4)) \notag\\
& \quad +\zeta_3((2,2,2,2,s,2),\lambda_3^\vee;SU(4)) +\zeta_3((2,2,2,2,2,s),\lambda_3^\vee;SU(4))\bigg\} \\
& = \left(372\cdot 2^{-s-10} + 306\right)\left(2^{-s-9}-1\right)\zeta(s+10)
+100\pi L(s+9,\chi_4)  \notag\\
& \quad +\left(7\cdot 2^{-s-8} + \frac{32}{3}\right)\left(2^{-s-7}-1\right)\pi^2\zeta(s+8)+ \frac{17}{6}\pi^3 L(s+7,\chi_4)   \notag\\
& \quad + \left(\frac{113\cdot 2^{-s-6}}{1440}+ \frac{1}{288}\right)\left(2^{-s-5}-1\right)\pi^4 \zeta(s+6) \notag\\
& \quad + \frac{1}{32}\pi^5 L(s+5;\chi_4)+\frac{289\cdot 2^{-s-4}}{241920}\left(2^{-s-3}-1\right)\pi^6 \zeta(s+4). \notag 
\end{align*}
Note that the reason why the right-hand sides of the above two formulas are the same is given by \eqref{change-1}. 
Setting $(a,b,c,p,q,s)=(2,2,2,2,2,s)$ in \eqref{Fq-PU4}, we obtain 
\begin{align*}
& 2\bigg\{\zeta_3((s,2,2,2,2,2),{\bf 0};PU(4))+\zeta_3((2,s,2,2,2,2),{\bf 0};PU(4)) \\
& \quad +\zeta_3((2,2,s,2,2,2),{\bf 0};PU(4)) +\zeta_3((2,2,2,s,2,2),{\bf 0};PU(4)) \notag\\
& \quad +\zeta_3((2,2,2,2,s,2),{\bf 0};PU(4)) +\zeta_3((2,2,2,2,2,s),{\bf 0};PU(4)) \bigg\} \notag\\
& =\left(-\frac{93\cdot 2^{-2s}}{262144} + \frac{33\cdot 2^{-s}}{512} + 306\right)\zeta(s+10) -50\pi L(s+9,\chi_4)\\
& \quad + \left(-\frac{7\cdot 2^{-2s}}{65536} - \frac{19\cdot 2^{-s}}{1536} - \frac{49}{3}\right)\pi^2\zeta(s+8) - \frac{17}{12}\pi^3 L(s+7,\chi_4)\\
& \quad + \left(-\frac{113\cdot 2^{-2s}}{5898240} - \frac{323\cdot 2^{-s}}{61440} + \frac{353}{576}\right)\pi^4\zeta(s+6) - \frac{\pi^5}{64} L(s+5,\chi_4)\\
& \quad + \left(-\frac{289\cdot 2^{-2s}}{61931520} -  \frac{31\cdot 2^{-s}}{7741440} + \frac{1}{720}\right)\pi^6\zeta(s+4). 
\end{align*}
In particular, setting $s=2$, we again obtain \eqref{val-PU4}.
\end{example}

\begin{example} \label{Ex-6-2} \rm 
Similarly to Example \ref{Exam-A3}, We obtain, for example, 
\begin{align*}
& \zeta_3((2,3,3,3,3,3),\lambda_1^\vee;SU(4))= \zeta_3((2,3,3,3,3,3),\lambda_3^\vee;SU(4)) \\
& \quad =\frac{2125\pi^8}{45097156608}\zeta(9) + \frac{11\pi^7}{15360}\,L(10,\chi_4) - \frac{440049247\pi^6}{225485783040}\zeta(11) \\
& \qquad + \frac{13\pi^5}{96}\,L(12,\chi_4) - \frac{1056786549\pi^4}{2147483648}\zeta(13) + 11\pi^3\,L(14,\chi_4) \\
& \qquad - \frac{199887481225\pi^2}{4294967296}\zeta(15) + 399\pi\,L(16,\chi_4) - \frac{2424501730875}{2147483648}\zeta(17),
\end{align*}
and 
\begin{align*}
& \zeta_3((2,3,3,3,3,3),\{{\bf 0}\};PU(4)) \\
& \quad =\frac{1111987\pi^{8}}{90194313216}\zeta(9) - \frac{11\pi^{7}}{30720}L(10,\chi_4) + \frac{11180759837\pi^{6}}{1352914698240}\zeta(11)  \\
& \qquad - \frac{13\pi^{5}}{192}L(12,\chi_4) + \frac{170862984923\pi^{4}}{64424509440}\zeta(13) - \frac{11\pi^{3}}{2}L(14,\chi_4)\\
& \qquad - \frac{1504872383333\pi^{2}}{25769803776}\zeta(15)  - \frac{399\pi}{2}L(16,\chi_4) + \frac{4849040457275}{4294967296}\zeta(17).
\end{align*}
\end{example}

\


\begin{thebibliography}{999}

\bibitem{Ap}
T. M. Apostol, \emph{Introduction to Analytic Number Theory}, Springer, 1976.

\bibitem{Bour}
N. Bourbaki, \emph{Groupes et Alg{\`e}bres de Lie, Chapitres 4, 5 et 6}, Hermann, 1968. 


\bibitem {GS}
{P. E. Gunnells and R. Sczech}, 
Evaluation of Dedekind sums, Eisenstein cocycles, and special values of $L$-functions, \emph{Duke Math.~J.} {\bf 118} (2003), 229--260.


\bibitem{Ho}
{M. E. Hoffman},
Multiple harmonic series, \emph{Pacific J. Math.} {\bf 152} (1992), 275-290. 


\bibitem{Iwa}
{K. Iwasawa},
\emph{Lectures on $p$-adic $L$-functions}, Princeton Univ. Press, 1972. 



\bibitem{KMT}
{Y. Komori, K. Matsumoto and H. Tsumura}, 
Zeta-functions of root systems, in ``The Conference on $L$-functions'' (Fukuoka, 2006), L. Weng and M. Kaneko (eds.), World Scientific, 2007, pp.~115--140.

\bibitem{KMTpja}
{Y. Komori, K. Matsumoto and H. Tsumura}, 
Zeta and $L$-functions and Bernoulli polynomials of root systems, {Proc. Japan Acad.}, Series A, {\bf 84} (2008), 57--62.

\bibitem{KMTJC}
{Y. Komori, K. Matsumoto and H. Tsumura}, 
Functional relations for zeta-functions of root systems, in ``Number Theory: Dreaming in Dreams - Proceedings of the 5th China-Japan Seminar'', T. Aoki, S. Kanemitsu and J. -Y. Liu (eds.), World Scientific Publ, 2010, pp.~135--183. 

\bibitem{KM5}
{Y. Komori, K. Matsumoto and H. Tsumura}, 
On multiple Bernoulli polynomials and multiple $L$-functions of root systems, {Proc. London Math. Soc.} {\bf 100} (2010), 303--347.

\bibitem{KMT-Mem}
{Y. Komori, K. Matsumoto and H. Tsumura}, 
An introduction to the theory of zeta-functions of root systems, in ``Algebraic and Analytic Aspects of Zeta Functions and $L$-functions'', G. Bhowmik, K. Matsumoto and H. Tsumura (eds.), MSJ Memoirs, Vol. 21, Mathematical Society of Japan, 2010, pp.~115--140. 

\bibitem{KM2}
{Y. Komori, K. Matsumoto and H. Tsumura}, 
On Witten multiple zeta-functions associated with semisimple Lie algebras II, {J. Math. Soc. Japan} {\bf 62} (2010), 355--394. 

\bibitem{KMT-Debrecen}
{Y. Komori, K. Matsumoto and H. Tsumura},
Functional equations and functional relations for the Euler double zeta-function and
its generalization of Eisenstein type,
{Publ. Math. Debrecen} {\bf 77} (2010), 15-31.

\bibitem{KM3}
{Y. Komori, K. Matsumoto and H. Tsumura}, 
On Witten multiple zeta-functions associated with semisimple Lie algebras III, to appear in ``Multiple Dirichlet Series, $L$-functions and Automorphic Forms'' (Proc. Edinburgh Conf., 2008), D. Bump et al. (eds.), Progr. Math., Birkh\"auser,\ arXiv:0907.0955. 

\bibitem{KM4}
{Y. Komori, K. Matsumoto and H. Tsumura}, 
On Witten multiple zeta-functions associated with semisimple Lie algebras IV, {Glasgow Math. J.} {\bf 53} (2011), 185-206.

\bibitem{KMT-Z}
{Y. Komori, K. Matsumoto and H. Tsumura},
Shuffle products for multiple zeta values and partial fraction decompositions
of zeta-functions of root systems,
{Math. Z.} {\bf 268} (2011), 993-1011.

\bibitem{KMT-Bessatsu}
{Y. Komori, K. Matsumoto and H. Tsumura},
A survey on the theory of multiple Bernoulli polynomials and multiple $L$-functions
of root systems,
in ``Infinite Analysis 2010, Developments in Quantum Integrable Systems'',
A. Kuniba et al. (eds.), RIMS K{\^o}ky{\^u}roku Bessatsu {\bf B28} (2011), 99-120.

\bibitem{KMT-Lie}
{Y. Komori, K. Matsumoto and H. Tsumura}, 
Zeta-functions of weight lattices of compact connected semisimple Lie groups, 
arXiv:math/1011.0323, submitted. 



\bibitem{MNO}
K. Matsumoto, T. Nakamura, H. Ochiai and H. Tsumura,
On value-relations, functional relations and singularities of Mordell-Tornheim and related triple zeta-functions, {Acta Arith.} {\bf 132} (2008), 99--125.

\bibitem{MT}
{K. Matsumoto and H. Tsumura,} 
On Witten multiple zeta-functions associated with semisimple Lie algebras I, {Ann. Inst. Fourier} {\bf 56} (2006), 1457--1504.

\bibitem{Naka}
T. Nakamura, Double Lerch value relations and functional relations
for Witten zeta functions, Tokyo J. Math. {\bf 31} (2008), 551-574.



\bibitem{Sz98}
A. Szenes, 
Iterated residues and multiple Bernoulli polynomials, 
{Internat. Math. Res. Notices}, {\bf 18} (1998), 937--958.

\bibitem{Sz03}
A. Szenes, 
Residue formula for rational trigonometric sums, 
{Duke Math. J.} {\bf 118} (2003), 189--228. 





\bibitem{TsC}
{H. Tsumura,} 
On functional relations between the Mordell-Tornheim double zeta functions and the Riemann zeta function, {Math.~Proc.~Cambridge Philos. Soc.}, {\bf 142} (2007), 395--405.


\bibitem{WW}
E. T. Whittaker and G. N. Watson, 
\emph{A Course of Modern Analysis}, 4th ed., Cambridge University Press, Cambridge, 1927.


\bibitem{Wi}
E. Witten, 
On quantum gauge theories in two dimensions, {Comm.~Math.~Phys.} {\bf 141} (1991), 153--209.

\bibitem{Za}
D. Zagier, 
Values of zeta functions and their applications, in ``First European Congress of Mathematics'' Vol.~II, A. Joseph
   et al. (eds.), Progr.~Math.~{120}, Birkh{\"a}user, 1994,
   pp.~497--512.

\end{thebibliography}
\end{document}